\documentclass[twoside]{article}
\usepackage{amsmath,amssymb,amsthm,bm,enumerate,color,geometry,mathrsfs}
\usepackage[dvips]{graphicx}
\usepackage[titletoc,title]{appendix}
\usepackage{lipsum,fancyhdr}

\newcommand{\keywords}[1]{\textbf{Key words and phrases:} #1}
\newcommand{\msc}[1]{\textbf{MSC2010:} #1}

\DeclareMathOperator{\Capacity}{Cap}

\def\rnum#1{\expandafter{\romannumeral #1}} 
\def\Rnum#1{\uppercase\expandafter{\romannumeral #1}}

\title{Rate functions for random walks on Random conductance models and related topics}
\author{Chikara Nakamura}
\date{June 29th, 2016}

\begin{document}
\newtheorem{Definition}{Definition}[section]
\newtheorem{Proposition}[Definition]{Proposition}
\newtheorem{Theorem}[Definition]{Theorem}
\newtheorem{Assumption}[Definition]{Assumption}
\newtheorem{Lemma}[Definition]{Lemma}
\newtheorem{Remark}[Definition]{Remark}
\newtheorem{Example}[Definition]{Example}
\newtheorem{Corollary}[Definition]{Corollary}
\newtheorem{Application}[Definition]{Application}
\newtheorem{TenAss}[Definition]{Tentative Assumption}

\newtheorem*{Notation}{Notation}
\newtheorem*{Acknowledgment}{Acknowledgment}

\makeatletter
 \renewcommand{\theequation}{%
   \thesection.\arabic{equation}}
  \@addtoreset{equation}{section}
\makeatother

\maketitle

\begin{abstract}
 We consider laws of the iterated logarithm and the rate function for sample paths of random walks on random conductance models under the assumption that 
   the random walks enjoy long time sub-Gaussian heat kernel estimates.   
\end{abstract}

\begin{flushleft}
   \\
 \msc{60J10, 60J35, 60F20.}  \\
 \end{flushleft}


\section{Introduction}    

The random conductance model (RCM) is a pair of a graph and a family of non-negative random variables (random conductances) 
    which are indexed by edges of the graph.  
The RCM includes various important examples such as the supercritical percolation cluster,
     whose random conductances are i.i.d. Bernoulli random variables. 
In the recent progress on the RCM, various asymptotic behaviors of random walks are obtained on a class of RCM 
     such as invariance principle, functional CLT, local CLT and  long time heat kernel estimates.
Here is a partial list of examples of the RCM;  
    \begin{enumerate} 
          \item  Uniform elliptic case \cite{Delmotte}, 
          \item  The supercritical percolation cluster \cite{Barlow1},  
          \item  I.i.d. unbounded conductance bounded from below \cite{BD},
          \item  I.i.d. bounded conductance under some tail conditions near $0$ \cite{BKM}, 
          \item  The level sets of Gaussian free field and the random interlacements \cite{Sapozhnikov}. 
    \end{enumerate}
We refer \cite{SS}, \cite{BB}, \cite{MP} for the invariance principle for random walks on the supercritical percolation cluster, 
 \cite{BH} for the local limit theorem for random walks on the supercritical percolation cluster, 
 \cite{ABDH} for the invariance principle on general i.i.d. RCMs, 
 \cite{ADS} for the Gaussian heat kernel upper bound on the possibly degenerate RCMs.   
We also refer \cite{Biskup} and \cite{Kumagai} for more details about the RCM.

In \cite{KN}, we discussed the laws of the iterated logarithms (LILs) for discrete time random walks 
     on a class of  RCM under the assumption of long time heat kernel estimates.  
The aims of this paper are to establish the laws of the iterated logarithm and to describe the rate functions 
       for the sample paths of continuous time random walks on the RCM.

The LILs describe the fluctuation of stochastic processes, which was originally obtained by Khinchin \cite{Khinchin} for a random walk. 
  We establish the LIL w.r.t. both $\displaystyle \sup_{0 \le s \le t} d(Y_0^{\omega}, Y_s^{\omega})$ and $d(Y_0^{\omega},  Y_t^{\omega})$, 
        and another LIL, which describes liminf behavior of $\displaystyle \sup_{0 \le s \le t} d(Y_0^{\omega}, Y_s^{\omega})$. 

The rate function describes the sample path ranges of stochastic processes.  
  For $d$-dimensional Brownian motion $B= \{ B_t \}_{t \ge 0}$, the Kolmogorov test tells us that 
         \begin{align*}   
                  \mathbb{P} \left( |B_t| \ge t^{1/2} h(t) \text{ for sufficiently large $t$ } \right) =
                            \begin{cases}
                                        1   \\  
                                        0,    
                            \end{cases} 
                  \quad \text{ according as }  \quad  \int_{1}^{\infty}  \frac{1}{t} h(t)^{d} e^{-\frac{h(t)^2}{2}} dt 
                            \begin{cases}
                                    < \infty \\  
                                    = \infty, 
                            \end{cases} 
           \end{align*} 
      where $h(t)$ is a positive function such that $h(t) \nearrow \infty$ as $t \to \infty$. 
  For $d\ge 3$, the Dvoretzky and Erd\H{o}s test tells us that 
         \begin{align}  \label{IntroRF2}
                  \mathbb{P} \left( |B_t| \ge t^{1/2} h(t) \text{ for sufficiently large $t$ } \right) =
                            \begin{cases}
                                        1   \\  
                                        0,    
                            \end{cases}
                  \quad \text{ according as }  \quad  \int_{1}^{\infty}  \frac{1}{t} h(t)^{d-2} dt 
                            \begin{cases}
                                    < \infty \\  
                                    = \infty, 
                            \end{cases} 
           \end{align} 
   where $h(t)$ is a positive function such that $h(t) \searrow 0$ as $t \to \infty$.  
 These results were extended to various frameworks such as symmetric stable processes on $\mathbb{R}^d$, 
    Brownian motions on Riemannian manifolds, symmetric Markov chains on weighted graphs and  $\beta$ stable like processes ($\beta \ge 2$).  
  We establish an analogue of \eqref{IntroRF2} w.r.t. random walks on the RCM.

Our approach is as follows; We assume quenched heat kernel estimates and establish both quenched LILs and an analogue of the Dvoretzky and Erd\H{o}s test. 
As we will see in Section \ref{Subsec:Ex}, our results are applicable for various models since heat kernel estimates are obtained for random walks on various RCMs.   
The concrete examples are given in Section \ref{Subsec:Ex}.  

The organization of this paper is as follows.  
  First, we give the framework and main results of this paper in Section  \ref{Subsec:FW} and examples in Section \ref{Subsec:Ex}. 
   In Section \ref{Sec:RFHK} we establish some preliminary results. 
   In Section \ref{Sec:LIL} we give the proof of the LILs.   
   In Section \ref{Sec:LRF} we establish an analogue of \eqref{IntroRF2}.  
   Finally in Section \ref{Sec:Erg} we discuss the case where $G=\mathbb{Z}^d$ and the media is ergodic.

In this paper, we use the following notation.  
\begin{Notation}
   \begin{enumerate}  \renewcommand{\labelenumi}{(\arabic{enumi})}     
          \item      We use $c, C, c_1 , c_2 , \cdots$ as the deterministic  positive constants. 
                      These constants do not depend on the random environment $\omega$, time parameters $t,s \cdots$,  distance parameters $r,\cdots$, and vertices of graphs. 

          \item    We define $a \vee b := \max \{ a,b \}$ and $a \wedge b := \min \{ a,b \}$.

   \end{enumerate}
\end{Notation}

\subsection{Framework and Main results}   \label{Subsec:FW}
  Let $G= (V,E) = (V(G), E(G))$ be a  countable and connected graph of bounded degree, i.e. 
      $\displaystyle M := \sup_{x \in V(G)} \deg x < \infty$.  
 We write $x \sim y $ if $(x,y) \in E(G)$.   
 A sequence $\ell_{xy} : x=x_0, x_1, \cdots , x_n = y$ on $G$ is called a path from $x$ to $y$ if $x_i \sim x_{i+1}$ for all $i=0,1, \cdots, n-1$. 
 We write $d(\cdot, \cdot)$ as the usual graph distance, that is, the length of a shortest path in $G$, and denote $B(x,r) = \{ y \in V(G) \mid d(x,y) \le r \}$. 

 Throughout of this paper we assume that there exist $\alpha \ge 1$, $c_1, c_2>0$ such that 
            \begin{gather}  \label{NumVer}
                        c_{1} r^{\alpha} \le \sharp B(x,r) \le c_{2} r^{\alpha} 
             \end{gather}
    for any $x \in V(G)$ and $r \ge 1$.

We introduce the random conductance model below. 
Let $\omega = \{ \omega_e = \omega_{xy} \}_{e= (x,y) \in E(G) }$ be a family of non-negative weight which is defined on a probability space  $(\Omega, \mathcal{F}, \mathbb{P})$. 
For  non-negative weights $\omega = \{ \omega_e \}_{e}$, we define $\displaystyle \pi^{\omega} (x) = \sum_{y;y \sim x} \omega_{xy}$ and $\displaystyle \nu^{\omega} (x) = 1$. 
We fix a base point $x_0 \in V(G)$, 
   and define graphs $G^{\omega} = (V(G^{\omega}), E(G^{\omega}))$ as 
            \begin{align*}
                &V(G^{\omega} ) = \left\{ y \in V(G) \middle| 
                             \begin{array}{l}
                                   \text{ There exists a path $\ell_{x_0y}:x_0, x_1, \cdots , x_n = y$ such that} \\ 
                                   \text{  $\omega_{x_i x_{i+1}} >0$ for all $i=0,1, \cdots, n-1$.} 
                            \end{array} \right\},   \\
                &E(G^{\omega}) = \{ e=(x,y) \in E(G) \mid x,y \in V(G^{\omega}) \text{ and } \omega_{xy} > 0 \}.
            \end{align*} 
 We denote $d^{\omega} (\cdot, \cdot)$ as the graph distance of $G^{\omega}$. 
 Note that $G^{\omega} = G$ and $d^{\omega} = d$ if conductance $\omega$ is strictly positive.

 We will consider two types of random walks, constant speed random walk (CSRW) and variable speed random walk (VSRW) associated to $\omega \in \Omega$.  
  Both CSRW and VSRW are continuous time random walk whose transition probability is given by $\displaystyle P^{\omega}(x,y) = \frac{ \omega_{xy} }{ \pi^{\omega} (x)}$. 
  For the CSRW, the holding time distribution at $x \in V(G^{\omega})$ is Exp $(1)$, whereas for the VSRW, the holding time distribution at $x \in V(G^{\omega})$ is Exp $(\pi^{\omega} (x))$.   
  We write $\mathcal{L}_{\theta}^{\omega}$ for the generator which is given by 
             \begin{align*}
                         \mathcal{L}_{\theta}^{\omega} f (x) = \frac{1}{\theta^{\omega} (x) }  \sum_{y;y \sim x} (f(y) - f(x)) \omega_{xy},
             \end{align*}
  and we also write the corresponding  heat kernel as  
             \begin{align*}
                       q_t^{\omega} (x,y) =  \frac{P^{\omega} (x,y)}{\theta^{\omega} (y) }, 
             \end{align*}
   where $\theta^{\omega} = \pi^{\omega}$ for the CSRW case and $\theta^{\omega} \equiv 1$ for the VSRW case.  
  We write $Y^{\omega} = \{Y_t^{\omega} \}_{t \ge 0}$ as either the CSRW or the VSRW, $P_x^{\omega}$ as the law of the random walk $Y^{\omega}$ which starts at $x$, and   
                    \begin{align}  \label{Hit}
                                \tau_F = \tau_F^{\omega} = \inf \{ t \ge 0 \mid Y_t^{\omega} \not\in F \},  
                                 \quad  \sigma_F = \sigma_F^{\omega} = \inf \{ t \ge 0 \mid Y_t^{\omega} \in F \},  
                                 \quad  \sigma_F^+ = \sigma_F^{+\omega} =  \inf \{ t > 0    \mid Y_t^{\omega} \in F \}.
                     \end{align} 
 We denote $F^{\omega} = F \cap V(G^{\omega})$,  $V^{\omega} (F) = \sum_{y \in F \cap V^{\omega} (G)} \theta^{\omega} (y)$ for $F \subset V(G)$ 
     and $V^{\omega} (x,r) = V^{\omega} (B(x,r))$.   
We write  $T_0^{\omega} = 0$ and $T_{n+1}^{\omega} = \inf \{ t>T_n^{\omega} \mid Y_t^{\omega} \neq Y_{T_n^{\omega}}^{\omega} \}$, and introduce 
  a discrete time random walk $\{ X_n^{\omega} := Y_{T_n^{\omega}}^{\omega} \}_{n \ge 0}$.    

First, we state the results about LILs. To do this, we need the following assumptions. 
\begin{Assumption}   \label{Ass10} 
There exist positive constants $\epsilon , \beta$ such that $\epsilon < \beta +1$ 
    and a family of non-negative random variables $\{ N_x = N_{x,\epsilon } \}_{x \in V(G)}$ such that the following hold; 
     \begin{enumerate}   \renewcommand{\labelenumi}{(\arabic{enumi})} 
          \item There exist positive constants $c_{1.1}, c_{1.2} , c_{1.3} , c_{1.4}$ such that 
                        \begin{align}  \label{UHK} 
                                 q_t^{\omega} (x,y)  &\le  
                                        \begin{cases}
                                                \frac{c_{1.1} }{ t^{\alpha/\beta} }  
                                                            \exp \left( - c_{1.2} \left( \frac{ d (x,y)^{\beta} }{t} \right)^{1/ (\beta - 1)}  \right),     
                                                                     &    \text{if $t \ge d(x,y)$,}    \\
                                                c_{1.3} \exp \left( - c_{1.4} d(x,y) \left( 1 \vee \log \frac{d(x,y)}{t} \right)  \right),   
                                                                    &    \text{if $t \le d(x,y)$,}
                                     \end{cases}      
                          \end{align}
                   for almost all $\omega \in \Omega$, all $x,y \in V(G^{\omega}) $ and $t\ge N_{x} (\omega )$. 

          \item  There exist positive constants $c_{2.1}, c_{2.2}$ such that
                     \begin{align}  \label{LHK}  
                          q_t^{\omega} (x,y) 
                                 &\ge  \frac{c_{2.1}}{t^{\alpha/\beta}} \exp \left( - c_{2.2} \left( \frac{ d(x,y)^{\beta} }{t} \right)^{1 / (\beta - 1)} \right)   
                     \end{align}
                    for almost all $\omega \in \Omega$, all $x,y \in V(G^{\omega})$ and 
                    $t \ge 0$ with $d(x,y)^{1+\epsilon } \vee N_{x} (\omega) \le t$.

          \item There exist positive constants $c_{3.1} , c_{3.2}$ such that 
                    \begin{align} \label{Vol}
                                c_{3.1} r^{\alpha}  \le V^{\omega} (x,r) \le  c_{3.2}  r^{\alpha}   
                    \end{align}
                    for almost all $\omega \in \Omega$, all $x \in V(G^{\omega}) $ and $r \ge N_x (\omega )$.    

          \item There exist positive  constants $c_{4.1}, c_{4.2}, c_{4.3}, c_{4.4}, c_{4.5}$ such that
                    \begin{align}  \label{CV}
                              & q_t^{\omega} (x,y) \le  
                                   \begin{cases}  
                                         \frac{c_{4.1}}{ \sqrt{\theta^{\omega} (x)  \theta^{\omega} (y)}}  \exp \left( -c_{4.2}  \frac{d(x,y)^2}{t} \right),
                                                     &  \text{ if $t \ge c_{4.3}d(x,y)$,}    \\  
                                          \frac{c_{4.4}}{ \sqrt{\theta^{\omega} (x)  \theta^{\omega} (y)}}  \exp \left( -c_{4.5} d(x,y) \left( 1 \vee \log \frac{d(x,y)}{t} \right) \right), 
                                                     &   \text{ if $t \le c_{4.3} d(x,y)$,}  
                                   \end{cases}
                    \end{align} 
                   for almost all $\omega \in \Omega$,  all $t >0$ and  $x,y \in V(G^{\omega})$ with $d(x,y) \ge N_x (\omega) \wedge N_y (\omega)$.

     \end{enumerate}
\end{Assumption}

\eqref{CV} is called Carne-Varopoulos bound. 
Note that \eqref{UHK} holds for $t \ge N_x (\omega)$ while \eqref{CV} holds for all $t >0$. 
It is known that \eqref{CV} holds under general conditions which will be described in the following Proposition (see \cite[Theorem 2.1, 2.2]{Folz}).  

\begin{Proposition}  \label{Prop:FPP}
Let $\{ N_x \}$ be as in Assumption \ref{Ass10} and $d_{\theta}^{\omega} (\cdot, \cdot)$ be a metric on $G^{\omega} = (V(G^{\omega}), E(G^{\omega}))$ which satisfies 
              \begin{align} \label{FPPmetric}
                      \frac{1}{ \theta^{\omega} (x)} \sum_{y \in V(G^{\omega})} d_{\theta}^{\omega} (x,y)^2 \omega_{xy} \le 1. 
              \end{align}  
 If there exists a positive constant $c$ such that $d_{\theta}^{\omega} (x,y) \ge c d(x,y)$ for all $x,y \in V(G^{\omega})$ with $d(x,y) \ge N_x(\omega) \wedge N_y(\omega)$, 
 then \eqref{CV} holds. 
\end{Proposition}

Next we assume the following three types of integrability conditions. 
\begin{Assumption} \label{Ass30}
  Let $\{ N_{x} \}_{x \in V(G)}$ be as in Assumption \ref{Ass10} and define $f (t) = f_{\epsilon} (t) =  \mathbb{P} (N_{x} \ge t ) $. 
  We consider the following three types of integrability conditions. 
         \begin{enumerate}  \renewcommand{\labelenumi}{(\arabic{enumi})} 
                    \item  $\displaystyle \sum_{n \ge 1} n^{\alpha} f (n)  < \infty$, 

                    \item   $\displaystyle \sum_{n \ge 1} n^{\alpha \beta} f (n)  < \infty$, 
                  
                    \item  For positive and non-increasing function $h(t)$,  $\displaystyle  \sum_{n} n^{\alpha } f (n h(n^{\beta}))  < \infty $.
          \end{enumerate} 
\end{Assumption}

We now state the main results of this paper. 
   \begin{Theorem}  \label{MainLIL}  
        \begin{enumerate}  \renewcommand{\labelenumi}{(\arabic{enumi})}   
               \item  Under Assumption \ref{Ass10} (1) (2) (3) and Assumption \ref{Ass30} (1),  for almost all $\omega \in \Omega$ there exists positive numbers 
                          $c_1 = c_1^{\omega}, c_2 = c_2^{\omega}$ such that 
                               \begin{align}   \label{MainLIL1}  
                                   \begin{split}
                                       & \limsup_{t \to \infty}  \frac{d(Y_0^{\omega}, Y_t^{\omega} )}{t^{1/\beta} (\log \log t)^{1-1/\beta} } 
                                               = c_1, \qquad  \text{$P_x^{\omega}$-a.s. for all $x \in V(G^{\omega})$},    \\
                                        & \limsup_{t \to \infty}  \frac{ \sup_{0\le s \le t} d(Y_0^{\omega}, Y_s^{\omega} )}{t^{1/\beta} (\log \log t)^{1-1/\beta} } 
                                               = c_2, \qquad  \text{$P_x^{\omega}$-a.s. for all $x \in V(G^{\omega})$}.     
                                  \end{split}
                               \end{align}  
     
               \item   Under Assumption \ref{Ass10} (1) (2) (3) and Assumption \ref{Ass30} (2),  for almost all $\omega \in \Omega$ there exist a positive number 
                          $c_3 = c_3^{\omega}$ such that 
                               \begin{align}
                                        \liminf_{t \to \infty} \frac{ \sup_{0 \le s \le t} d(Y_0^{\omega}, Y_s^{\omega} )}{t^{1/\beta} (\log \log t)^{-1/\beta} } 
                                                = c_3, \qquad  \text{$P_x^{\omega}$-a.s. for all $x \in V(G^{\omega})$} .  \label{MainLIL2}     
                               \end{align}  
          \end{enumerate} 
  \end{Theorem}

   \begin{Theorem}    \label{Thm:LRF} 
      Suppose Assumption \ref{Ass10} (1) (2) (3) (4) and $\alpha/\beta >1$.
       In addition $\theta^{\omega} (x) = \pi^{\omega} (x) \ge c$ for a  positive constant $c>0$ in the case of CSRW.
      Let $h:(1,\infty) \rightarrow (0,\infty)$ be a function such that $h(t) \searrow 0$ as $t \to \infty$ 
             and the function $\varphi (t) := t^{1/\beta} h(t)$ is increasing. 
      If  $h(t)$ satisfies Assumption \ref{Ass30} (3) and 
                   \begin{align*}
                             \int_1^{\infty} \frac{1}{t} h(t)^{\alpha - \beta} dt < \infty  \text{ or } = \infty
                   \end{align*}
        then 
                   \begin{align*}
                          P_x^{\omega} \left( d(x,Y_t^{\omega}) \ge t^{1/\beta} h(t) \text{ for all sufficiently large $t$} \right) = 1  \text{ or } 0.
                   \end{align*}
   \end{Theorem}

Finally we discuss the constants $c_1, c_2, c_3$ in \eqref{MainLIL1} and \eqref{MainLIL2}. 
When we consider a case of $G=\mathbb{Z}^d$, we can take $c_1, c_2$ as deterministic constants under some appropriate assumptions.   
To state this, we take the base point $x_0 = 0 \in \mathbb{Z}^d$ and we write shift operators as $\tau_x, (x \in \mathbb{Z}^d)$, where $\tau_x$ is given by 
      \begin{gather}  \label{shift}
                ( \tau_x \omega)_{yz} = \omega_{x+y, x+z}.
      \end{gather}
We assume the following conditions. 
 
\begin{Assumption} \label{Ass50}
     Assume that $(\Omega, \mathcal{F},\mathbb{P})$ satisfies the following conditions;
                  \begin{enumerate}
                            \item[(1)]       $\mathbb{P}$ is ergodic with respect to the translation operators $\tau_x$, namely $\mathbb{P} \circ \tau_x = \mathbb{P}$ and 
                                                if $\tau_x (A) = A$ for all $x \in \mathbb{Z}^d$ and for all $A \in \mathcal{F}$ then $\mathbb{P} (A) = 0$ or $1$. 
                            
                             \item[(2)]  For almost all environment $\omega$, $V(G^{\omega})$ contains a unique infinite connected component.  
 
                            \item[(3)] (VSRW case) $\displaystyle \mathbb{E} \left[ \frac{1}{\pi^{\omega} (0)} \right] \in (0,\infty)$. 
                 \end{enumerate}

    \end{Assumption}

\begin{Theorem}  \label{Thm:Const}
   Suppose that the same assumptions as in Theorem \ref{MainLIL} and suppose in addition Assumption \ref{Ass50}.  
     Then we can take $c_1, c_2, c_3$ in \eqref{MainLIL1} and \eqref{MainLIL2}  as deterministic constants (i.e. do not depend on $\omega$).
 \end{Theorem}

\subsection{Example}  \label{Subsec:Ex} 
 In this subsection, we give some examples for which our results are applicable. 
    
\begin{Example}[Bernoulli supercritical percolation cluster]  
Let $G=(\mathbb{Z}^d, E_d)$ be a graph, where $E_d = \{ \{ x,y \} \mid x,y \in \mathbb{Z}^d, |x-y|_{1} = 1 \}$.   
Put a Bernoulli random variable $\omega_e$ with $\mathbb{P} (\eta_e = 1) = p$ on each edge. 
This model is called bond percolation. We write $p_c (d)$ as the critical probability. 
It is known that there exists a unique infinite connected component when $p > p_c (d)$.  
See \cite{Grimmett} for more details about the percolation. 

Barlow \cite{Barlow1} proved that heat kernels of CSRWs on the super-critical percolation cluster (that is,  when $p > p_c (d)$)  
    on $\mathbb{Z}^d$, $d\ge 2$ satisfy Assumption \ref{Ass10} (1) (2) (3) (4), Assumption \ref{Ass30} (1) (2)
    with $\alpha = d$, $\beta = 2$ and $f_{\epsilon} (t) = c \exp (-c^{\prime} t^{\delta})$ for some $c,c^{\prime},  \delta>0$.
Since the media is i.i.d. and there exists an unique infinite connected component, we can obtain Theorem \ref{MainLIL} with deterministic constants by Theorem \ref{Thm:Const}.

In addition, we can easily check that $\displaystyle h(t) = \frac{1}{(\log t)^{ \kappa / (d-2)}}$ for $\kappa > 0$ satisfy the conditions in Assumption \ref{Ass30} (3)  and
  the assumptions of Theorem \ref{Thm:LRF} in the case of $d>2$.  
Thus $P_x^{\omega} \left( d(x,Y_t^{\omega}) \ge t^{1/\beta} h(t) \text{ for all} \right.$ $\left. \text{sufficiently large $t$ } \right) = 1,0$ 
      according as $\kappa >d-2, \le d-2$ respectively by Theorem \ref{Thm:LRF}. 

Note that \eqref{MainLIL1} for the supercritical percolation cluster was already obtained by \cite[Theorem 1.1]{Duminil-Copin}. 
\end{Example}

\begin{Example}[Gaussian free fields and random  interlacements]  
Gaussian free field on a graph $G = (V,E)$ is a family of centered Gaussian variables $\{ \varphi_x \}_{x \in G}$ with covariance $E[ \varphi_x \varphi_y ] = g(x,y)$, 
   where $g(x,y)$ is the Green function of a random walk on $G$.
Here we are interested in the level sets of the Gaussian free field $E_h = \{ x \in V \mid \varphi_x \ge h \}$.  
We can regard the level sets as one of the percolation models which has correlation among the vertecies in $V$. 
See \cite{SznitmanText} for the details. 
  
The random interlacements concern geometries of random walk trajectories, e.g. how many random walk trajectories are needed to make the underlying graph disconnected?  Sznitman \cite{Sznitman} formulated the model of random interlacements. 
However the model of random interlacements is defined through Poisson point process on a trajectory space, 
 we can also regard this model as the percolation model with long range correlation. From the viewpoint of the RCM, we can regard the model of random interlacements as 
 one of the RCM whose conductances take the value $0$ or $1$ and the conductances are not independent.  
See \cite{DRS} for the details.

Sapozhnikov  \cite[Theorem 1.15]{Sapozhnikov} proved that for $\mathbb{Z}^d$, $d\ge 3$, 
 the CSRWs on (i) certain level sets of Gaussian free fields;  
(ii) random interlacements at level $u>0$; (iii) vacant sets of random interlacements for suitable level sets, satisfy our Assumption \ref{Ass10} (1) (2) (3) 
with $\alpha = d$, $\beta = 2$ and the tail estimates of $N_{x} (\omega)$ as $f_{\epsilon} (t) = c\exp (-c^{\prime} (\log t)^{1+\delta} )$ for some $c,c^{\prime}, \delta >0$.   
As the same reason with the case of Bernoulli supercritical percolation cluster, Assumption \ref{Ass10} (3) is also satisfied in these models.  
This subexponential tail estimate is sufficient for Assumption \ref{Ass30} $(3)$ with $\displaystyle h(t) = \frac{1}{(\log t)^{ \kappa / (d - 2)}}$ for $\kappa > 0$. 
Since the media is ergodic and 
there is an unique infinite connected components (see \cite{RS}, \cite[Corollary 2.3]{Sznitman} and \cite[Theorem 1.1]{Teixeira}), 
Theorem \ref{MainLIL} holds with deterministic constants by Theorem \ref{Thm:Const}, and Theorem \ref{Thm:LRF} holds with $\displaystyle h(t) = \frac{1}{(\log t)^{ \kappa / (d-2)}}$ for $\kappa \ge d-2$ or $<d-2$ respectively.
\end{Example}

\begin{Example}[Uniform elliptic case]
Suppose that a graph $G=(V,E)$ is endowed with weight $1$ on each edge and satisfies \eqref{NumVer} and the scaled Poincar\'e inequalities.  
Take $c_1, c_2 $ as positive constants and put random conductances on all edges so that $c_1\le \omega (e) \le c_2$ for all $e\in E$ and for almost all $\omega$. 
Delmotte \cite{Delmotte} obtained Gaussian heat kernel estimates for CSRWs in this framework. 
Thus Assumption \ref{Ass10} (1) (2) (3) hold with $\beta = 2$ and $N_{x,\epsilon} \equiv 1$.  
Hence Theorem \ref{MainLIL}  holds. 

In addition, Assumption \ref{Ass10} is followed by \cite[Corollary 11, 12]{Davies}.  
(See also Proposition \ref{Prop:FPP}, note that  the graph distance satisfies \eqref{FPPmetric} for CSRW case.) 
Thus Theorem \ref{Thm:LRF} holds with $\displaystyle h(t) = \frac{1}{(\log t)^{ \kappa / (d-2)}}$  ($\kappa \ge d-2$ or $< d-2$ respectively).
\end{Example}

\begin{Example}[Unbounded conductance bounded from below]   \label{Ex:BD}
Let $G = \mathbb{Z}^d$ $(d \ge 2)$ and put random conductances $\omega = \{ \omega_{xy} \}_{xy \in E}$ which take the value $[1,\infty)$. 
 Barlow and Deuschel \cite[Theorem 1.2]{BD} proved that the heat kernels of VSRW satisfy Assumption \ref{Ass10} (1) (2), Assumption \ref{Ass30} (1) (2) with 
  $\alpha = d$, $\beta = 2$ and $f_{\epsilon} (t) = c_1\exp (-c_2 t^{\delta})$ for some $c_1, c_2 , \delta > 0$. 
 (Note that Assumption \ref{Ass10} (3) is trivial since $V^{\omega} (x,r) = \sharp B(x,r)$ for the VSRW.) 
Hence Theorem \ref{MainLIL} holds. 

In addition, Assumption \ref{Ass10} (4) is followed by \cite[Theorem 2.3, Theorem 4.3 (b)]{BD} or  \cite[Theorem 2.1, Theorem 2.2]{Folz}. 
Thus  Theorem \ref{Thm:LRF} for the VSRW holds with $\displaystyle h(t) = \frac{1}{(\log t)^{ \kappa / (d-2)}}$  ($\kappa \ge d-2$ or $< d-2$ respectively). 

Moreover, if the conductances $\{ \omega_e \}_e$ satisfy Assumption \ref{Ass50} (3) then Theorem \ref{MainLIL} holds with deterministic constants.     
\end{Example}

\section{Consequences of Assumption \ref{Ass10}}   \label{Sec:RFHK}

In this section we give some preliminary results of our assumptions. 

\subsection{Consequences of heat kernel estimates} 
In this subsection, we give preliminary results of Assumption \ref{Ass10} (1) (2) (3).  

Recall the notations in $\eqref{UHK}$.  
\begin{Lemma}  \label{RFHK10}
  Suppose Assumption \ref{Ass10} (1) (3). 
  For all $\delta  \in (0, c_{1.2} \wedge c_{1.4})$ there exist positive constants $c_1 = c_1 (\delta), c_2 = c_2 (\delta), c_3 = c_3 (\delta )$ such that
          \begin{align}  \label{RFHK11}
                            P_x^{\omega} \left( d(x,Y_t^{\omega} )  \ge r \right)   
                                 \le c_1 \exp \left[ -(c_{1.2} - \delta )  \left( \frac{r}{t^{1/\beta}} \right)^{\frac{\beta}{\beta -1}}  \right]  
                                  + c_2  \exp \left( - c_3 t \right)   
          \end{align} 
 for almost all $\omega \in \Omega$,  all $x \in V(G^{\omega})$, $r \ge N_{x} (\omega )$ and $t \ge N_{x} (\omega )$.
\end{Lemma}

This lemma is standard except for the part of estimates of Poissonian regime (the bottom line of \eqref{UHK}). 
For completeness I give the proof here. 

\begin{proof} 
We first prepare some preliminary facts to estimate $P_x^{\omega} \left( d(x,Y_t^{\omega} ) \ge r \right)$.  
Set $\displaystyle h_1 (\eta , s) = \exp \left[ - \eta s^{\beta/(\beta -1)} \right]$ and $h_2 (\eta , s) = \exp \left[ -\eta s \right]$.   
     For $h_1 (\eta, s)$, we can easily see that there exists a constant $\zeta_0 > 1$ such that 
                    \begin{align}   \label{RFHK17}
                               h_1(\eta, \zeta s ) \le h_1(\eta ,1) h_1(\eta, s) 
                    \end{align}
         for all $\zeta \ge \zeta_0$, $\eta >0$ and $s \ge 1$. (We can take $\zeta_0$ as the positive number which satisfies $\zeta_0^{\beta/(\beta -1)} -1 = 1$.)  
     For $h_2 (\eta , s)$, we can easily see that 
                    \begin{align}    \label{RFHK15}
                                h_2 (\eta, \zeta s) \le h_2 (\eta , 1) h_2 (\eta ,s)
                    \end{align}
     for all $\zeta \ge 2$, $\eta > 0$ and $s \ge 1$. 
    Next, we easily see that for all $\zeta > 1$ there exists $c_1 = c_1 (\zeta)$ such that for almost all $\omega \in \Omega$
                    \begin{align}   \label{RFHK12}
                                V^{\omega}  (x, r \zeta) \le c_1 V^{\omega}(x,r)
                    \end{align}  
       for all $x \in V(G)$ and for all $r \ge N_{x} (\omega )$.  (Use \eqref{Vol} and take $c_1 = \frac{c_{3.2} \zeta^{\alpha} }{c_{3.1}}$.) 
   Thirdly, it is also easy to see that for all $\delta \in (0,c_{1.2})$ there exists $c_2 (\delta )$ such that 
                    \begin{align}  \label{RFHK13}
                                 s^{\alpha} \exp \left[ - c_{1.2} s^{\beta/(\beta -1)} \right] \le c_2 (\delta )  \exp \left[ - (c_{1.2} - \delta ) s^{\beta / (\beta - 1)} \right]
                    \end{align}
       for all $s \ge 1$, where $c_{1.2}$ is the same constant as in \eqref{UHK}. 
       We can  also see that for all $\delta \in (0,c_{1.4}) $ there exists a positive constant $c_3 = c_3 (\delta )$ such that 
                    \begin{align}  \label{RFHK16}
                               s^{\alpha} \exp \left[ - c_{1.4} s \right]  \le  c_3 (\delta )    \exp \left[ - (c_{1.4} - \delta ) s \right] 
                    \end{align}
      for all $s \ge 1$. 
  Using \eqref{RFHK13},  we can see that for $d(x,z) \ge s \ge t^{1/\beta}$ and $\delta \in (0,c_{1.2})$  
                    \begin{align}  \label{RFHK14}
                                 & \frac{c_{1.1}}{t^{\alpha /\beta}}  \exp \left[ - c_{1.2} \left(  \frac{d(x,z)}{t^{1/\beta}} \right)^{\beta / (\beta -1)}  \right]  
                                                = \frac{c_{1.1} }{ d(x,z)^{\alpha} }  \left(  \frac{d(x,z)}{t^{1/\beta}} \right)^{\alpha}  
                                                     \exp \left[ -c_{1.2} \left(  \frac{ d(x,z)}{ t^{1/\beta} } \right)^{\beta/ (\beta -1)} \right]  \notag \\
                                 &  \le  \frac{ c_4 (\delta )}{ d(x,z)^{\alpha} } \exp \left[ - (c_{1.2} - \delta )  \left( \frac{d(x,z)}{t^{1/\beta}} \right)^{\beta / (\beta -1)} \right]   
                                                \quad \left( \text{use \eqref{RFHK13}} \right)   \notag \\
                                 &  \le  \frac{ c_4 (\delta )}{ s^{\alpha}}  \exp \left[  - (c_{1.2}- \delta ) \left(  \frac{s}{t^{1/\beta}} \right)^{\beta / (\beta - 1)} \right],  
                                                \quad \left( \text{use $d(x,z) \ge s$} \right).
                    \end{align}

Now we estimate $P_x^{\omega} (d(x,Y_t^{\omega}) \ge r) $.
We first consider the case $r \le t^{1/\beta}$.
Since $s \mapsto h_1 (\eta, s)$, ($\eta > 0$) is non-increasing, we have  
                     \begin{align}  \label{RFHK24}
                                 P_x^{\omega} (d(x,Y_t^{\omega}) \ge r)   \le 1 \le \frac{h_1 \left( c_{1.2}, \frac{r}{t^{1/\beta}} \right)}{h_1 (c_{1.2},1) } = c_5h_1 \left( c_{1.2}, \frac{r}{t^{1/\beta}} \right) ,
                     \end{align}    
   where we set $c_5 = 1/h(c_{1.2}, 1)$.  
So we may and do assume $r \ge t^{1/\beta}$. 
Take $\zeta \ge \zeta_0 \vee 2$ so that 
           \eqref{RFHK17}, \eqref{RFHK15} and \eqref{RFHK12} hold.
 We divide $P_x^{\omega} (d(x,Y_t^{\omega}) \ge r) $ into   
                    \begin{align}   \label{RFHK19}
                                \sum_{k=0}^K  \sum_{z \in B^{\omega} (x, r\zeta^{k+1} ) \setminus B^{\omega} (x, r \zeta^k) }  q_{t}^{\omega} (x,z) \theta^{\omega} (z)  ,
                               \qquad    \sum_{k=K}^{\infty}    \sum_{z \in B^{\omega} (x, r\zeta^{k+1} ) \setminus B^{\omega} (x, r \zeta^k) }  q_{t}^{\omega}(x,z) \theta^{\omega} (z) ,             
                    \end{align}
           where $K$ is the positive integer which satisfies $r\zeta^K \le t < r \zeta^{K+1}$. 
We have for $t \ge N_{x} (\omega )$, $r \ge N_{x} (\omega )$ and using \eqref{UHK} 
                   \begin{align}  \label{RFHK20} 
                            & (\text{The first term of \eqref{RFHK19}})
                                   \le  \sum_{k=0}^K  \sum_{z \in B^{\omega} (x,r\zeta^{k+1}) \setminus B^{\omega} (x,r\zeta^k) }
                                   \frac{ c_{1.1}}{t^{\alpha / \beta} }  \exp \left[ -c_{1.2}  \left(   \frac{ d(x,z) }{ t^{1/\beta} } \right)^{\beta / (\beta - 1)}  \right]  
                                       \theta^{\omega} (z)  \notag \\
                            & \le   \sum_{k=0}^K  \frac{c_6 (\delta )}{ (r\zeta^k)^{\alpha} }  
                                         \exp \left[  - (c_{1.2} - \delta ) \left(  \frac{r \zeta^k}{t^{1/\beta}} \right)^{\beta / (\beta -1)}  \right]  (r\zeta^{k+1})^{\alpha}
                                         \qquad  \left( \text{use \eqref{RFHK14} and \eqref{Vol} }  \right)   \notag \\
                            &   \le   \sum_{k=0}^K   c_7 (\delta , \zeta )  h_1 \left( c_{1.2} - \delta , \frac{r \zeta^k}{t^{1/\beta}} \right)    \notag \\
                            &    \le  c_7 (\delta , \zeta )  h_1 \left( c_{1.2} - \delta , \frac{r}{t^{1/\beta}}  \right)    \sum_{k=0}^K    h_1 (c_{1.2} - \delta , 1)^k  
                                         \qquad   \left( \text{use \eqref{RFHK17}} \right)     \notag    \\
                            &   \le  c_8 (\delta , \zeta )   \exp \left[ - (c_{1.2} - \delta ) \left( \frac{r}{t^{1/\beta}}  \right)^{\beta /(\beta - 1)}  \right],  
                                          \qquad (\text{since $h_1 (c_{1.2} - \delta, 1) < 1$} ). 
                  \end{align}
    For the second term of \eqref{RFHK19}, using \eqref{UHK}, $t \ge N_{x} (\omega )$ and $r \ge N_{x} (\omega )$ we have 
                  \begin{align}    \label{RFHK23}
                            &  (\text{The second term of \eqref{RFHK19}})
                                \le \sum_{k=K}^{\infty}   \sum_{z \in B^{\omega} (x,r\zeta^{k+1}) \setminus  B^{\omega} (x,r\zeta^k)} 
                                        c_{1.3}  \exp \left[ - c_{1.4}  d(x,z) \left(  1 \vee \log \frac{d(x,z)}{t}  \right)   \right]  \theta^{\omega} (z)   \notag   \\
                            &   \le  \sum_{k=K}^{\infty}   \sum_{z \in B^{\omega} (x,r\zeta^{k+1}) \setminus  B^{\omega} (x,r\zeta^k)} 
                                        c_{1.3}  \exp \left[ - c_{1.4}  d(x,z)  \right]   \theta^{\omega} (z)    
                                        \quad  \left(  \text{since $1 \vee \log \frac{d(x,z)}{t}  \ge 1$}   \right)   \notag   \\      
                            &  \le   \sum_{k=K}^{\infty}   c_9   \exp \left[ - c_{1.4} (r\zeta^k) \right]  (r\zeta^{k+1})^{\alpha}  
                                         \qquad  \left(  \text{use \eqref{Vol}}  \right)        \notag    \\  
                            &  \le c_{10} (\zeta, \delta )  \sum_{k=K}^{\infty}  \exp \left[- (c_{1.4}- \delta ) r \zeta^k  \right]  
                                          \qquad  \left(  \text{use \eqref{RFHK16}} \right)   \notag  \\
                            &    =  c_{10} (\zeta, \delta )  \sum_{k=K}^{\infty}  h_2 \left( c_{1.4}- \delta ,  r \zeta^k  \right)    \notag    \\
                            &   \le c_{11} (\zeta , \delta )  h_2 (c_{1.4} -\delta, r\zeta^K )   \sum_{k=0}^{\infty} h_{2} (c_{1.4} - \delta ,1)^k   
                                         \qquad   \left(  \text{use \eqref{RFHK15}}  \right)  \notag  \\
                            &   \le c_{12} (\zeta , \delta ) \exp  \left[ - c_{13} (\zeta, \delta)t \right],   
                                       \qquad \left( \text{since $r\zeta^K \le t < r\zeta^{K+1}$}  \right) .    
                  \end{align}
  
Therefore, by \eqref{RFHK24}, \eqref{RFHK20}, \eqref{RFHK23} and adjusting the constants, we obtain \eqref{RFHK11}. 
We thus complete the proof.
\end{proof}

Again recall the notations $c_{1.2}$ and $c_{1.4}$ in \eqref{UHK}. 
\begin{Lemma}   \label{RFHK30} 
  Suppose Assumption \ref{Ass10} (1) (3).  
  For all $\delta \in (0, c_{1.2} \wedge c_{1.4})$ there exist positive constants $c_1 = c_1(\delta ), c_2 =c_2 (\delta ), c_3 = c_3 (\delta)$ such that 
        \begin{align}   
                        & P_x^{\omega} \left(  \sup_{0 \le s \le t} d(x, Y_s^{\omega} )  \ge 2r \right)  
                              \le c_1  \exp \left[ - (c_{1.2} - \delta ) \left( \frac{r}{(2t)^{1/\beta}}  \right)^{\beta/(\beta -1)} \right]  + c_2 \exp \left[ -c_3 t \right]  \label{RFHK31-1}   \\  
                        & P_x^{\omega} \left(  \sup_{0 \le s \le t} d(y, Y_s^{\omega} )  \ge 4r \right)  
                              \le c_1  \exp \left[ - (c_{1.2} - \delta ) \left( \frac{r}{(2t)^{1/\beta}}  \right)^{\beta/(\beta -1)} \right]  + c_2 \exp \left[ -c_3 t \right]     \label{RFHK31-2}  
          \end{align}
  for almost all $\omega \in \Omega$, all $x,y \in V(G^{\omega})$, $t \ge 1$ and $r \ge 1$ with 
  $d(x,y) \le 2r$, $\displaystyle t \ge \max_{u \in B(x,2r)} N_{u} (\omega ) $ and $\displaystyle r \ge \max_{u \in B (x,2r)} N_{u} (\omega )$.
\end{Lemma}

\begin{proof}
This is standard (see the proof of \cite[Lemma 3.9 (c)]{Barlow2}), so we omit the proof.
\end{proof}

\begin{Lemma} \label{RFHK80} 
  Suppose Assumption \ref{Ass10} (1) (2) (3). 
Then there exist positive constants $\eta \ge 1, c_1, c_2>0$ such that
      \begin{align}  \label{RFHK81} 
              P_x^{\omega} \left(  \sup_{0 \le s \le t} d(x,Y_s^{\omega} ) \le 3 \eta r \right)  \ge c_1 \exp \left[  -c_2 \frac{t}{r^{\beta}} \right]
      \end{align}
  for almost all $\omega \in \Omega$, all $x \in V(G^{\omega})$, $t \ge r \ge 1$ with 
  $\displaystyle  r^{1/\beta} \ge \max_{z \in B(y,3\eta r)} N_{z} (\omega) $. 
\end{Lemma}

\begin{proof}
 The proof is quite similar to that of \cite[Proposition 3.3]{KN2}, so we omit the proof. 
\end{proof}

 Let $c_1,c_2$ be as in Lemma \ref{RFHK80}. Note that we can assume that $c_1 < 1$ (and therefore $c_1 \exp [ -c_2] \in (0,1)$). 
 We define $\rho_1, a_k, b_k, \lambda_k, u_k, \sigma_k$ as 
        \begin{equation}  
             \begin{split}  \label{Not10} 
                 & \rho_1 = c_1 \exp [ -c_2], ~~ a_k^{\beta} = e^{k^2}, ~~   b_k^{\beta} = e^k,     \\
                 &\lambda_k = \frac{2}{3 |\log \rho_1 |} \log (1+k), ~~    u_k = \lambda_k a_k^{\beta}, ~~   \sigma_k = \sum_{i=1}^{k-1} u_i .
             \end{split}
        \end{equation}

\begin{Corollary}[Corollary of Lemma \ref{RFHK80} ]    \label{RFHK90}
 Let $\eta \ge 1$ be as in Lemma \ref{RFHK80}. 
 Then under Assumption \ref{Ass10} (1) (2) (3) we have
                         \begin{align}   \label{RFHK92}
                                    \inf_{z \in B(x,a_k ) }  P_z^{\omega }  \left(  \sup_{0 \le s \le u_k} d(z,Y_s^{\omega} )  \le 3\eta a_k  \right)  \ge \rho_1^{\lambda_k}
                         \end{align}
    for almost all $\omega \in \Omega$, all $k$ with $\displaystyle \max_{z \in B(x, 4\eta a_k )} N_{v} (\omega ) \le a_k^{1/\beta}$.
\end{Corollary}

\begin{proof} 
 We can see from  Lemma \ref{RFHK80} that 
               \begin{align*}
                                    P_z^{\omega }  \left(  \sup_{0 \le s \le u_k} d(z,Y_s^{\omega} )  \le3 \eta  a_k  \right)  
                                         \ge  c_1 \exp \left[ - c_2 \frac{u_k}{a_k^{\beta} } \right]  \ge  \rho_1^{\lambda_k}  
               \end{align*}
   for all $k\ge 1$ with $\displaystyle \max_{v \in B(z, 3 \eta a_k )} N_{v} (\omega ) \le a_k^{1/\beta}$. 
   Hence \eqref{RFHK92} holds for $k$ with $\displaystyle \max_{z \in B(x,a_k) } \max_{v \in B(z, 3 \eta a_k  )} N_{v} (\omega ) \le a_k^{1/\beta}$.  
     
\end{proof}

\begin{Lemma}  \label{RFHK70} 
 Suppose Assumption \ref{Ass10} (1) (3). Then there exist positive  constants $c_1, c_2$ such that 
                   \begin{align*}
                            P_x^{\omega} \left(  \sup_{0\le s \le t} d(x,Y_s^{\omega}) \le  r \right)  
                                   \le c_1 \exp \left( - c_2 \frac{t}{r^{\beta}} \right)
                    \end{align*}
   for  almost all environment $\omega \in \Omega$, all $x \in V(G^{\omega})$, $t\ge 1$ and $r \ge 1$ with $\displaystyle \max_{y \in B(x,r)}  N_{y} (\omega )   \le 2r$. 
 \end{Lemma}

\begin{proof} 
The proof is quite similar to that of \cite[Lemma 3.2]{KN2}, so we omit it.
\end{proof}

We will need the following version of 0-1 law. 
\begin{Theorem}[$0-1$ law for tail events]  \label{RF0-1-30}  
For almost all environment $\omega \in \Omega$, the following holds;  
Let $A^{\omega}$ be a tail event, i.e. $\displaystyle  A^{\omega} \in \bigcap_{t=0}^{ \infty }  \sigma \{ Y_s^{\omega} : s \ge t \} $.
Then either  $P_x^{\omega} (A^{\omega})  = 0 $ for all $x $ or $P_x^{\omega} (A^{\omega}) = 1$ for all $x$. 
\end{Theorem} 

The proof of the above theorem is quite similar to that of \cite[Proposition 2.3]{BK} (see also \cite[Theorem 4]{Barlow1}), so we omit the proof here. 

\subsection{Green function}
In this subsection,  we deduce Green function estimates. 
We define Green function as 
                \begin{align}  \label{Green}
                            g^{\omega} (x,y) = \int_0^{\infty} q_t^{\omega} (x,y) dt. 
                \end{align}

Recall that $\theta^{\omega} (x) = \pi^{\omega} (x)$ in the case of CSRW and $\theta^{\omega} (x) = 1$ in the case of VSRW.  
 \begin{Proposition}  \label{CV30}  
 Let $\alpha > \beta$ and suppose  Assumption \ref{Ass10} (1) (2) (4). In addition we assume  
    there exists a positive constant $c>0$ such that $\theta^{\omega} (x) \ge c$ for all $x \in V(G^{\omega})$ in the case of CSRW.  
 Then there exist positive constants $c_1, c_2$ such that 
             \begin{align}  \label{CV31}
                      \frac{c_1}{d(x,y)^{\alpha - \beta}} \le  g^{\omega} (x,y) \le \frac{c_2}{d(x,y)^{\alpha - \beta} }  
             \end{align}
      for almost all $\omega \in \Omega$, all $x,y \in V(G^{\omega})$ with $d(x,y) \ge N_{x} (\omega ) \wedge N_{y} (\omega)$. 
 \end{Proposition}

\begin{proof} 
 This proof is similar to  \cite[Proposition 6.2]{BH}. 
  We first prove the upper bound of \eqref{CV31}. 
             \begin{align}  \label{CV32}
                         &  g^{\omega} (x,y)   \notag \\  
                         &   = \int_{0}^{ (c_{4.3}d(x,y))\wedge N_{x} (\omega) }  q_t^{\omega} (x,y) dt   
                                             + \int_{ (c_{4.3}d(x,y))\wedge N_{x} (\omega)}^{ N_{x} (\omega) }  q_t^{\omega} (x,y) dt  
                                            +  \int_{N_{x} (\omega)}^{d(x,y)}  q_t^{\omega} (x,y) dt 
                                            +   \int_{d(x,y)}^{\infty}  q_t^{\omega} (x,y) dt   \notag \\ 
                         & := J_1 + J_2 + J_3 + J_4.
             \end{align} 
  We estimate $J_1, J_2, J_3, J_4$ as follows.
             \begin{align}    \label{CV33} 
                     \begin{split} 
                             J_1 &\le  \int_{0}^{ (c_{4.3}d(x,y)) \wedge N_{x} (\omega) }  \frac{c_{4.4}}{ \sqrt{\theta^{\omega} (x) \theta^{\omega} (y)} } \exp \left[ - c_{4.5} d(x,y) \right]  dt   
                                           \qquad \left( \text{use \eqref{CV}} \right)  \\ 
                                    & \le c_{1} d(x,y)  \exp \left[ - c_{2} d(x,y) \right],  \\  
                             J_2 &\le  \int_{ (c_{4.3}d(x,y)) \wedge N_{x} (\omega)}^{ N_{x} (\omega)} 
                                             \frac{c_{4.1}}{\sqrt{\theta^{\omega} (x) \theta^{\omega} (y)}} \exp \left[ -c_{4.2} \frac{d(x,y)^{2}}{t} \right] dt  
                                            \qquad \left( \text{use \eqref{CV}} \right)     \\
                                    & \le c_{3} N_{x} (\omega ) \exp \left[ -c_{4} \frac{d(x,y)^{2}}{N_{x} (\omega)} \right] 
                                    \le  c_{3} d(x,y) \exp \left[ -c_{4} d(x,y) \right]  \quad (\text{ use $d(x,y) \ge N_{x} (\omega)$}),   \\
                             J_3 &\le  \int_{N_{x} (\omega)}^{d(x,y)}   c_{1.3} \exp \left[ - c_{1.4} d(x,y) \right]  dt   
                                          \qquad \left( \text{use \eqref{UHK}} \right)  \\ 
                                  &\le  c_{1.3} d(x,y)  \exp \left[ - c_{1.4} d(x,y) \right],   \\
                            J_4  &\le  \int_{d(x,y)}^{\infty}  \frac{c_{1.1}}{t^{\alpha/\beta}} \exp \left[ - c_{1.2} \left( \frac{d(x,y)}{t^{1/\beta}} \right)^{\beta/(\beta -1)} \right] dt
                                           \le \frac{ c_{5} }{ d(x,y)^{\alpha - \beta} }. 
                       \end{split}
             \end{align}
 By \eqref{CV32} and \eqref{CV33} we have $\displaystyle g^{\omega} (x,y) \le \frac{c_{6}}{d(x,y)^{\alpha - \beta} }  $ for $d(x,y) \ge N_{x} (\omega)$. 
    Note that $g^{\omega} (x,y) = g^{\omega} (y,x)$. Thus we complete the upper bound of \eqref{CV31}.

Next we prove the lower bound of \eqref{CV31}.  We can obtain the lower bound in the following way.
           \begin{align*}  
                  g^{\omega} (x,y) \ge \int_{d(x,y)^{\beta}}^{\infty} q_t^{\omega} (x,y) dt  
                         \ge  \int_{d(x,y)^{\beta}}^{\infty}   \frac{c_{2.1}}{t^{\alpha / \beta}} \exp \left[ - c_{2.2} \left( \frac{d(x,y)}{t^{1/\beta}} \right)^{\beta/(\beta -1)} \right] dt 
                         \ge \frac{c_{7}}{d(x,y)^{\alpha -\beta}} .
           \end{align*}
  We thus complete the proof. 
\end{proof}

\subsection{Consequences of Green function and Assumption \ref{Ass10}}  \label{Sec:Green}   
In this subsection we give some preliminary results of Assumption \ref{Ass10} (1) (2) (3) (4)  in the case of $\alpha > \beta$. 
This subsection is based on \cite[Section 4.1]{SW}. 
In this subsection we assume the following conditions. 

\begin{Assumption} \label{Ass:Gr}
     \begin{enumerate}    \renewcommand{\labelenumi}{(\arabic{enumi})}   
              \item    $\alpha > \beta$, 
              \item  (CSRW case) There exists a positive constant $c$ such that $\theta^{\omega} (x) \ge c$ for all $x \in V(G^{\omega})$. 
     \end{enumerate}  
\end{Assumption}
Recall that Proposition \ref{CV30} holds under Assumption \ref{Ass10} (1) (2) (4) and Assumption \ref{Ass:Gr}.

We write $\displaystyle  e_F^{\omega} (x) =  P_x^{\omega} \left( \sigma_{F}^{+ \omega} = \infty \right) 1_F (x) $
   as the equilibrium measure of $F \subset V(G^{\omega})$, 
   and define  $\Capacity^{\omega} (F) = \sum_{x \in F} e_F^{\omega} (x) \theta^{\omega} (x)$ as the capacity of $F \subset V(G^{\omega})$. 
 Then we have 
        \begin{align}  \label{RFGr15}
               P_x^{\omega} \left( \sigma_F^{+\omega}  < \infty \right)  = \sum_{y \in F} g^{\omega} (x,y) e_F^{\omega} (y) \theta^{\omega} (y)
        \end{align}
  for any finite set $F$ and for any $x \in V(G^{\omega})$ since 
        \begin{align*} 
               &  P_x^{\omega} \left( \sigma_F^{+\omega}  < \infty \right)   
                       =  \int_{0}^{\infty}  \sum_{y\in F} P_x^{\omega} \left( Y_t^{\omega} = y, Y_s^{\omega} \not\in F \text{ for any $s > t$} \right) dt 
                       \quad  (\text{last exit decomposition}) \\ 
               & = \int_{0}^{\infty} \sum_{y \in F} q_t^{\omega} (x,y) \theta^{\omega} (y) P_y^{\omega} \left( \sigma_{F}^{+\omega} = \infty  \right) dt 
                        \quad (\text{by the Markov property} )    \\
               & = \sum_{y \in F} g^{\omega} (x,y) e_F^{\omega} (y) \theta^{\omega} (y).
        \end{align*}

\begin{Lemma}  \label{RFGr50}  
 Under Assumption \ref{Ass10} (1) (2) (3) (4) and Assumption \ref{Ass:Gr}, there exists a positive constant $c$ such that 
       \begin{align*}
              \Capacity^{\omega} ( B^{\omega} (x, 2r) ) \ge  c r^{\alpha - \beta}
       \end{align*}
 for almost all $\omega \in \Omega$, all $x \in V(G^{\omega})$ and $r \ge 1$ with $\displaystyle r \ge \max_{v \in B(x,r)} N_{v} (\omega)$. 
\end{Lemma}

\begin{proof} 
Recall the notations in \eqref{Hit}. 
       \begin{align*}
             1   &= \frac{1}{\theta^{\omega} (B(x,r))}  \sum_{y \in B^{\omega} (x,r)} P_y^{\omega} \left( \sigma_{B(x,2r)}^{+\omega} < \infty \right)  \theta^{\omega} (y)   \\  
                 &= \frac{1}{\theta^{\omega} (B(x,r))}  \sum_{y \in B^{\omega}(x,r)} \sum_{ \substack{ z \in B^{\omega} (x,2r) \\ d(x, z)=2r} }
                              g^{\omega} (y,z) e_{B^{\omega}(x,2r)}^{\omega} (z)  \theta^{\omega} (z) \theta^{\omega} (y)  
                              \qquad  \left( \text{we use \eqref{RFGr15}} \right)  \\ 
                 & \le  \frac{c_1}{\theta^{\omega} (B(x,r))}  \frac{1}{r^{\alpha - \beta}} 
                            \sum_{ \substack{ z \in B^{\omega} (x,2r) \\ d(x, z)=2r} }\sum_{y \in B^{\omega}(x,r)}  e_{B^{\omega}(x,2r)}^{\omega} (z) \theta^{\omega} (z) \theta^{\omega} (y)  \\
                      & \quad \left( \text{ since $d(y,z) \ge r \ge N_y (\omega)$ and Proposition \ref{CV30} }\right) \\  
                 & =   \frac{c_1}{\theta^{\omega} (B(x,r))}  \frac{ \theta^{\omega} (B(x,r)) }{r^{\alpha - \beta}}  
                            \sum_{ \substack{ z \in B^{\omega} (x,2r) \\ d(x, z)=2r} } e_{B^{\omega} (x,2r)}^{\omega} (z) \theta^{\omega} (z)\\
                 & =   \frac{c_1}{r^{\alpha - \beta}}  \Capacity^{\omega} ( B^{\omega}(x,2r) ) . 
       \end{align*} 
 We thus complete the proof. 
\end{proof}

 Recall the notations in \eqref{Hit} and set 
   \begin{align*}
                    & \gamma_{x,F}^{\omega} (K_1) = P_x^{\omega} \left( Y_{\sigma_F^+}^{\omega} \in K_1 \right),  \\
                    & \pi_{x,F}^{\omega} (dt, K_2) = P_x^{\omega} \left( Y_{\sigma_F^+}^{\omega} \in K_2, \sigma_F \in dt \right)
   \end{align*}
  for $F,K_1,K_2 \subset V(G^{\omega})$. 
  Note that $\displaystyle \int_0^{\infty}  \pi_{x,F}^{\omega}  (dt, K) = \gamma_{x,F}^{\omega} (K)$ 
      and $\gamma_{x,F}^{\omega} (F) = P_x^{\omega} \left( \sigma_F^{\omega +} < \infty \right)$.

\begin{Lemma}   \label{RFGr20}
For almost all $\omega \in \Omega$, 
      \begin{align} \label{RFGr21}
                 g^{\omega} (x,y) = \sum_{v \in F^{\omega}} g^{\omega} (v,y)  \gamma_{x,F^{\omega} }^{\omega} (v)
      \end{align}
 for any finite set $F^{\omega} \subset V(G^{\omega})$, $x \not\in F^{\omega}$ and $y \in F^{\omega}$. 
 In particular we have 
      \begin{align}  \label{RFGr22}
                P_x^{\omega} \left( Y_t^{\omega} \in F^{\omega} \text{ for some $t >0$} \right)  
                         \le \inf_{y \in F^{\omega} }  \left(  \frac{ g^{\omega} (x,y)}{ \inf_{z \in F^{\omega}} g^{\omega} (z,y)} \right). 
      \end{align}
\end{Lemma}

\begin{proof}
 We write $F= F^{\omega}$ and $\sigma = \sigma_{F^{\omega}}^{\omega +} = \inf \{ t >0 \mid Y_t^{\omega}  \in F \}$ for notational simplification. 
 Then for any $ x \not\in F$, $y \in F$ we have  
       \begin{align*}
                & P_x^{\omega} \left( Y_t^{\omega} = y \right) 
                         = E_x^{\omega} \left[ 1_{ \{ \sigma \le t \} } P_{Y_\sigma^{\omega} }^{\omega} \left( Y_{t -\sigma}^{\omega} = y \right) \right]  
                         = \sum_{v\in F}  E_x^{\omega} \left[ 1_{ \{ \sigma \le t \} } 1_{ \{ Y_{\sigma}^{\omega} = v \} } 
                                 P_{Y_\sigma}^{\omega} \left( Y_{t -\sigma}^{\omega} = y \right) \right] \\
                & = \sum_{v \in F} \int_{0}^t P_v^{\omega} \left[ Y_{t-s}^{\omega} = y \right] \pi_{x,F}^{\omega} (ds, v) .
       \end{align*} 
Hence we have 
       \begin{align*}
             & g^{\omega} (x,y) = \int_0^{\infty} \sum_{v \in F} \int_{0}^t q_{t-s}^{\omega} (v,y)  \pi_{x,F}^{\omega} (ds,v) dt
                =  \int_0^{\infty} \sum_{v \in F}  \int_{s}^{\infty} q_{t-s}^{\omega} (v,y) dt  \pi_{x,F}^{\omega} (ds,v)   \\ 
             & = \int_0^{\infty} \sum_{v \in F} g^{\omega} (v,y) \pi_{x,F}^{\omega}  (ds, v) 
                 = \sum_{v \in F} g^{\omega} (v,y)  \gamma_{x,F}^{\omega} (v). 
       \end{align*} 
 We thus complete the proof of \eqref{RFGr21}. \eqref{RFGr22} is immediate from \eqref{RFGr21}. 
\end{proof}

\begin{Lemma}  \label{RFGr30} 
  Under Assumption \ref{Ass10} (1) (2) (3) (4) and Assumption \ref{Ass:Gr} there exist positive constants $c_1, c_2$ such that for almost all $\omega \in \Omega$ the following hold. 
             \begin{enumerate}  \renewcommand{\labelenumi}{(\arabic{enumi})} 
                        \item  $\displaystyle  P_x^{\omega} \left( \sigma_{B(x_0,2r)}^{+\omega} < \infty \right)  
                                           \le    c_{1} \frac{r^{\alpha -\beta} }{(d(x,x_0) -r)^{\alpha -\beta} }$ 
                                   for all $x,x_0 \in V(G^{\omega})$, $r \ge 1$ with $d(x,x_0) \ge 2r + 1$ and   
                                  $\displaystyle r \ge \max_{v \in B(x_0,r)} N_{v} (\omega )$.   

                        \item  $\displaystyle   P_x^{\omega} \left( \sigma_{B(x_0,2r)}^{+\omega} < \infty \right)  
                                            \ge     c_{2} \frac{r^{\alpha -\beta} }{(d(x,x_0) +2r)^{\alpha -\beta} }  $ 
                                   for all $x,x_0 \in V(G^{\omega})$, $r \ge 1$ with $d(x,x_0) \ge 2r$, $r \ge N_x (\omega)$ and $\displaystyle r \ge \max_{v \in B(x_0,r)} N_v (\omega)$.  
                 \end{enumerate}
  
\end{Lemma}

\begin{proof}
We first prove (1) by using \eqref{RFGr22}.  
Let $x, x_0 \in V(G^{\omega})$ satisfy $d(x,x_0) \ge 2r +1$. 
  For any $y \in B(x_0,r)$ we have 
       \begin{align*}
                    d(x,y) \ge d(x,x_0) - d(x_0,y)  \ge d(x,x_0) -r  \ge 2r -r = r.
       \end{align*} 
  By Proposition \ref{CV30}, for any $y \in B^{\omega} (x_0,r)$ and for any $r$ with $\displaystyle r \ge \max_{y \in B(x_0,r)} N_y (\omega) $ we have 
            \begin{align} \label{RFGr32}
                   g^{\omega} (x,y) \le \frac{c_{1}}{d(x,y)^{\alpha - \beta}} \le  \frac{c_1}{ (d(x,x_0) - r)^{\alpha - \beta}}.
            \end{align}
 Next note that $B(x_0,2r) \subset B(y,3r)$ for any $y \in B(x_0,r)$.   
     Since $g^{\omega} (\cdot, y)$ is superharmonic function, using minimal principle and Proposition \ref{CV30} we have      
             \begin{align}  \label{RFGr33}
                   \inf_{z \in B^{\omega} (x_0,2r)}  g^{\omega} (z,y) \ge \inf_{z \in B^{\omega} (y,3r)} g^{\omega} (z,y) 
                            \ge \inf_{ \substack{z \in B^{\omega}(y,3r+1) \\ d(y,z) = 3r+1}} g^{\omega}(z,y)
                            \ge \frac{c_2}{r^{\alpha -\beta}}
      \end{align}
    for all $r \ge 1$ and $y \in B^{\omega} (x_0,r)$ with $\displaystyle 3r +1 \ge \max_{v \in B(x_0,r)} N_v (\omega )$.   
 Hence by  \eqref{RFGr22}, \eqref{RFGr32} and \eqref{RFGr33} we have 
      \begin{align*}
                  P_x^{\omega} \left( \sigma_{B(x_0,2r)}^{+} < \infty \right)  
                      \le \inf_{y \in B^{\omega} (x_0,r)} \left(  \frac{ g^{\omega} (x,y) }{ \inf_{z \in B^{\omega} (x_0,2r)} g(z,y)} \right)
                      \le  c_{3} \frac{r^{\alpha -\beta} }{(d(x,x_0) -r)^{\alpha -\beta} }
      \end{align*}
    for all $r$ with $\displaystyle r \ge \max_{v \in B(x_0,r)} N_v (\omega )$.   
 Thus we complete the proof of $(1)$.  

Next we prove (2).  Note that  
          \begin{align*}
                   &  P_x^{\omega} \left( \sigma_{B(x_0,2r)}^{+\omega} < \infty \right) 
                            = \sum_{y\in B^{\omega} (x_0,2r)} g^{\omega} (x,y) e_{B^{\omega} (x_0,2r)}^{\omega} (y) \theta^{\omega} (y)  
                            \qquad (\text{use \eqref{RFGr15}})  \\
                   &  \ge \left( \inf_{y \in B^{\omega} (x_0,2r)} g^{\omega} (x,y) \right) 
                              \sum_{y \in B^{\omega} (x_0,2r)}  e_{B(x_0,2r)}^{\omega} (y) \theta^{\omega} (y)     
                              = \left( \inf_{y \in B^{\omega} (x_0,2r)} g^{\omega} (x,y) \right)   \Capacity^{\omega} (B(x_0,2r) )  . 
           \end{align*} 
  By $B(x_0,2r) \subset B(x,d(x,x_0) + 2r)$, the minimum principle for superharmonic functions and our assumptions we have
           \begin{align*}
                     \inf_{y \in B^{\omega} (x_0,2r)} g^{\omega} (x,y) 
                          \ge \inf_{y \in B^{\omega} (x, d(x,x_0) +2r)} g^{\omega} (x,y) 
                          \ge \inf_{ \substack{ y \in B^{\omega} (x,d(x,x_0) + 2r+1) \\ d(y,x) = d(x,x_0) +2r+1}} g^{\omega} (x,y) 
                          \ge \frac{c_{4}}{ ( d(x,x_0) + 2r )^{\alpha - \beta} } 
            \end{align*} 
    for $r \ge N_x (\omega)$. 
  By  Lemma \ref{RFGr50} $\displaystyle   \Capacity^{\omega} (B(x_0,r) ) \ge c_5 r^{\alpha - \beta}$ for 
        $\displaystyle r \ge \max_{v \in B(x_0,r)} N_v (\omega)$. Hence 
           \begin{align*}
                     &P_x^{\omega} \left( \sigma_{B(x_0,2r)}^{+\omega} < \infty \right) 
                               \ge \frac{c_{6} r^{\alpha - \beta}}{ ( d(x,x_0) + 2r )^{\alpha - \beta} } 
           \end{align*}
   for $r \ge N_x (\omega)$ and $\displaystyle  r \ge \max_{v \in B(x_0,r)} N_v (\omega)$.  
 We thus complete the proof. 
 
\end{proof}

\begin{Lemma}  \label{RFGr70} 
   Under Assumption \ref{Ass10} (1) (2) (3) (4) and Assumption \ref{Ass:Gr} there exist positive constants $c_1$ and $T_0$ such that 
           \begin{align*} 
                   P_x^{\omega} \left( d(x_0, Y_s^{\omega}) \le 2r \text{ for some $s>t$} \right) \le \frac{c_1 r^{\alpha - \beta}t }{t^{\alpha / \beta}}  
           \end{align*} 
 for almost all $\omega \in \Omega$, all $t\ge T_0$, $r \ge 1$ and  $x,x_0 \in V(G^{\omega}) $ with $t^{1/\beta} \ge r$, $d(x,x_0) \le r$ 
 and $\displaystyle r \ge \max_{z \in B(x_0, r)} N_z (\omega)$. 
\end{Lemma}

\begin{proof}  
  First note that 
       \begin{align*} 
              & P_x^{\omega} \left( d(x_0, Y_s^{\omega})  \le 2r \text{ for some $s>t$} \right) 
                        = \sum_{y \in V(G^{\omega}) } P_x^{\omega} \left( Y_t^{\omega} = y \right) 
                                  P_y^{\omega} \left( d(x_0, Y_s^{\omega}) \le 2r \text{ for some $s >0$} \right) \\
              &=  \sum_{y; t^{1/\beta} < d(x_0,y)-r}  P_x^{\omega} \left( Y_t^{\omega} = y \right)  
                            P_y^{\omega} \left( d(x_0, Y_s^{\omega}) \le 2r \text{ for some $s >0$} \right)  \\
              &~~  + \sum_{y; r < d(x_0,y) -r \le t^{1/\beta}} P_x^{\omega} \left( Y_t^{\omega} = y \right)  
                            P_y^{\omega} \left( d(x_0, Y_s^{\omega}) \le 2r \text{ for some $s >0$} \right)  \\ 
              &~~  + \sum_{y;d(x_0,y) \le 2r}  P_x^{\omega} \left( Y_t^{\omega} = y \right)  
                            P_y^{\omega} \left( d(x_0, Y_s^{\omega}) \le 2r \text{ for some $s >0$} \right)  \\
              &:= J_1 + J_2 + J_3. 
        \end{align*}
We estimate $J_1, J_2$ and $J_3$ in the following way.

For $t,r \ge 1$ with $t \ge N_{x} (\omega)$ and $\displaystyle r \ge \max_{z \in B(x_0,r)} N_{z}$ (note that $t \ge N_x (\omega)$ follows from our assumptions),    
  using \eqref{UHK}, Lemma \ref{RFGr30}, \eqref{Vol} we have 
         \begin{align*} 
               J_1 &\le \sum_{y; t^{1/\beta} < d(x_0,y)-r}  \frac{c_1 r^{\alpha -\beta} }{ (d(y,x_0) -r)^{\alpha - \beta} }  \\
                             &  \qquad \cdot   \left\{  \frac{c_{1.1} }{t^{\alpha / \beta}} 
                                             \exp \left[ -c_{1.2} \left( \frac{d(x,y)}{t^{1/\beta}} \right)^{ \frac{\beta}{\beta -1}} \right] 
                                             + c_{1.3} \exp \left[ -c_{1.4} d(x,y) \right]  \right\} \theta^{\omega} (y)  \quad (\text{use \eqref{UHK} and Lemma \ref{RFGr30}}) \\  
                   &\le  \sum_{\ell = 1}^{\infty}  \sum_{y; d(x_0,y) \in [ \ell t^{1/\beta} +r, (\ell +1) t^{1/\beta}+r ] }   \frac{c_{2} r^{\alpha -\beta} }{ (d(y,x_0) -r)^{\alpha - \beta} }  
                                       \frac{1 }{t^{\alpha / \beta}} \exp \left[ -c_{1.2} \left( \frac{d(y,x_0) - r}{t^{1/\beta}} \right)^{ \frac{\beta}{\beta -1}} \right]  
                                       \theta^{\omega} (y)   \\
                            &  \qquad +   \sum_{\ell = 1}^{\infty}  \sum_{y; d(x_0,y) \in [ \ell t^{1/\beta} +r, (\ell +1) t^{1/\beta} +r] }  
                                       \frac{c_{3} r^{\alpha -\beta} }{ (d(y,x_0) -r)^{\alpha - \beta} }  
                                        \exp \left[ -c_{1.4} (d(y,x_0) - r) \right]   \theta^{\omega} (y)      \\  
                              &  \qquad \left( \text{since $d(x,y) \ge d(y,x_0) - d(x_0,x)$ and $d(x_0,x) \le r$} \right) \\   
                    & \le  \sum_{\ell = 1}^{\infty}   \frac{c_{2} r^{\alpha -\beta} }{ (\ell t^{1/\beta})^{\alpha - \beta} }  \frac{1}{t^{\alpha / \beta}}
                                       \exp \left[ - c_{1.2} \ell^{\frac{\beta}{\beta -1}} \right] \theta^{\omega} \left( B(x_0, (\ell + 1)t^{1/\beta} +r) \right)  \\
                             &  \qquad +  \sum_{\ell = 1}^{\infty}   \frac{c_{3} r^{\alpha -\beta} }{ (\ell t^{1/\beta})^{\alpha - \beta} }  \exp \left[ - c_{1.4} \ell t^{1/\beta} \right] 
                                              \theta^{\omega} \left( B(x_0, (\ell + 1)t^{1/\beta} +r) \right)   \\
                    &\le \frac{ c_4 r^{\alpha - \beta}}{t^{\alpha / \beta -1}} \sum_{\ell =1}^{\infty} 
                                  \ell^{\beta} \exp \left[ - c_{1.2} \ell^{\beta / \beta -1} \right]   
                                  + \frac{ c_5 r^{\alpha - \beta}}{t^{\alpha / \beta -1}}  t^{\alpha / \beta}  \sum_{\ell = 1}^{\infty}  
                                           \ell^{\beta}   \exp \left[ - c_{1.4} \ell t^{1/\beta} \right]  \\  
                               &\qquad (\text{use $\theta^{\omega} ( B(x_0, (\ell + 1)t^{1/\beta} +r) )\le c (\ell t^{1/\beta})^{\alpha}$ since $t^{1/\beta} \ge r$ })    \\
                    &\le \frac{ c_6 r^{\alpha - \beta}}{t^{\alpha / \beta -1}},  
                               \qquad (\text{since $t \mapsto t^{\alpha / \beta}  \sum_{\ell = 1}^{\infty}  \ell^{\beta}   \exp \left[ - c_{1.4} \ell t^{1/\beta} \right]  $ is bounded} ).   
         \end{align*}

Next we see $J_2$. 
First, set  $\phi_r (k) = (r+k)^{\beta} (k,r \ge 1)$. We can easily see that there exist a positive constant $K_1 = K_1 (r) >1$ such that 
               \begin{align}  \label{RFGr72}
                                        \phi_r (k) \le \frac{1}{2} \phi_r (K_1 k)
               \end{align}
     for all $k \ge 1$.  Using this inequality we see that for $r \ge N_{x_0} (\omega)$ 
               \begin{align}   \label{RFGr73}
                                  & \sum_{y;r \le d(x_0,y) -r \le t^{1/\beta}}  \frac{\theta^{\omega} (y) }{ (d(y,x_0)-r)^{\alpha - \beta} } 
                                          \le \sum_{k=2r}^{r+t^{1/\beta}}  \frac{ \theta^{\omega} (B(x_0,k) \setminus B(x_0,k-1)) }{ (k-r)^{\alpha - \beta} }  \notag  \\
                                  & \le c_7 + \sum_{\ell = 1}^{(t^{1/\beta} - r)/K_1} \sum_{k \in [2r+ \ell K_1, 2r + (\ell + 1)K_1]} 
                                              \frac{ \theta^{\omega} (B(x_0,k) \setminus B(x_0,k-1)) }{ (k-r)^{\alpha - \beta} }   
                                          \le c_7 + c_8 \sum_{\ell = 1}^{(t^{1/\beta} - r)/K_1}  (r + \ell K_1)^{\beta}  \notag   \\
                                  & \le  c_7 + c_9 t, \qquad (\text{use \eqref{RFGr72}}).
               \end{align} 
 We go back to estimate $J_2$. Note that for $y$ with $r \le d(x_0,y) -r \le t^{1/\beta}$ we see $d(x,y) \le d(x,x_0) + d(x_0,y) \le 3t^{1/\beta}$. 
For $r \ge 1$, $t \ge 1$ with  $t \ge T_0 := 3^{\beta / (\beta -1)}$ (so that $3t^{1/\beta}\le t$ for $t \ge T_0$) 
  and $\displaystyle r \ge \max_{z \in B(x_0,r)} N_{z} (\omega)$ (in particular $t \ge N_{x} (\omega)$),
  using Lemma \ref{RFGr30}, \eqref{UHK} and \eqref{RFGr73} we have 
        \begin{align*}  
                 J_2 &\le \sum_{y; r < d(x_0,y)-r \le t^{1/\beta}} 
                             \frac{c_{10} r^{\alpha - \beta} }{(d(y,x_0) -r)^{\alpha - \beta}} \frac{\theta^{\omega}(y)}{t^{\alpha/\beta}} 
                             =  \frac{c_{10} r^{\alpha - \beta} }{t^{\alpha/\beta}}   \sum_{y; r\le d(x_0,y)-r \le t^{1/\beta}}   \frac{\theta^{\omega} (y)}{(d(y,x_0) -r)^{\alpha - \beta}} \\
                       &  \le  \frac{c_{11} r^{\alpha - \beta} t}{t^{\alpha / \beta}},  \qquad (\text{use \eqref{RFGr73}}). 
        \end{align*}

Finally we see $J_3$. For $t \ge T_0 := 3^{\beta/(\beta-1)}$, $N_{x} (\omega ) \le t$ and $N_{x} (\omega) \le r$, using \eqref{UHK} we have  
        \begin{align*}
                J_3  &\le  \sum_{y;d(y,x_0) \le 2r}   P_x^{\omega} \left( Y_t^{\omega} = y \right)  
                                 =   \sum_{y;d(y,x_0) \le 2r}  q_t^{\omega} (x,y) \theta^{\omega} (y)    \\  
                      &\le  \sum_{y; d(x,y) \le 3r} q_t^{\omega} (x,y) \theta^{\omega} (y) 
                                 \le \frac{c_{12} r^{\alpha}}{t^{\alpha / \beta}} 
                                \le \frac{c_{12} r^{\alpha - \beta} t}{t^{\alpha / \beta}}.
        \end{align*}
We thus complete the proof. 
\end{proof}

\begin{Lemma}  \label{RFGr80}
  Under Assumption \ref{Ass10} (1) (2) (3) (4) and Assumption \ref{Ass:Gr}  there exist constants $c_1>0, c_2 , T_0 \ge 1$ such that 
         \begin{align*}
                   P_x^{\omega} \left( d(x_0,Y_s^{\omega} ) \le 2r \text{ for some $s > t$ } \right)  \ge  \frac{c_1r^{\alpha - \beta} t}{t^{\alpha / \beta}}  
         \end{align*}
     for almost all $\omega \in \Omega$, all $r \ge 1$, $t \ge T_0$, $x,x_0 \in V(G^{\omega})$ with $d(x,x_0) \le r$, $t \ge r^{\beta}$,  
     $\displaystyle r \ge \max_{z \in B(x_0,c_{2} t^{1/\beta})} N_z (\omega)$. 
\end{Lemma}

\begin{proof}
Take a constant $c_2$ such that $c_{3.1} c_2^{\alpha} - c_{3.2}2^{\alpha} > 0$.  
  Note that by \eqref{Vol} we have $\theta^{\omega} (\{ y \in V(G) \mid d(x_0,y) \in [2t^{1/\beta},c_2t^{1/\beta}] \} ) \ge (c_{3.1}c_2^{\alpha} - c_{3.2} 2^{\alpha}) t^{\alpha / \beta} $, 
  and for $y$ and sufficiently large $t$ (say $t\ge T_0)$ with $d(x_0,y) \in [2t^{1/\beta}, c_2t^{1/\beta}]$ we have 
  $d(x,y)^{1+\epsilon} \le (d(x,x_0) + d(x_0,y) )^{1+\epsilon} \le \{ (c_2+1)t^{1/\beta} \}^{1+\epsilon} \le t$ since $1+\epsilon < \beta$ (see Assumption \ref{Ass10}). 
Then by Lemma \ref{RFGr30} (2), \eqref{LHK}, \eqref{Vol}, for $t,r$ as in the statement above we have 
        \begin{align*}  
                 & P_x^{\omega} \left( d(x_0,Y_s^{\omega} ) \le 2r \text{ for some $s > t$ } \right)  
                         = \sum_{y \in V(G^{\omega})} q_t^{\omega} (x,y) \theta^{\omega} (y) P_y^{\omega} \left( d(x_0,Y_s^{\omega}) \le 2r \text{ for some $s>0$ } \right)  \\ 
                 & \ge \sum_{y: d(x_0,y) \in [2 t^{1/\beta}, c_{2} t^{1/\beta} ] }  
                               q_t^{\omega} (x,y) \theta^{\omega} (y) P_y^{\omega} \left( d(x_0,Y_s^{\omega}) \le 2r \text{ for some $s>0$ } \right)   \\  
                 & \ge  \sum_{y: d(x_0,y) \in [2 t^{1/\beta}, c_{2} t^{1/\beta} ] }   \frac{c_{2.1}}{t^{\alpha / \beta}} 
                               \exp  \left[ - c_{2.2}  \left( \frac{d(x,y)}{t^{1/\beta}} \right)^{\beta/(\beta -1)} \right]
                               \theta^{\omega} (y) \frac{c_{3} r^{\alpha - \beta}}{ ( d(x_0,y) + 2r)^{\alpha - \beta} }    \\  
                         & \qquad  \left( \text{use \eqref{LHK}, Lemma \ref{RFGr30} and $d(x,y)^{1+\epsilon} \le t$, note that $t \ge N_x (\omega)$ follows from our assumptions} \right)  \\
                 & \ge  \sum_{y: d(x_0,y) \in [2 t^{1/\beta}, c_{2}t^{1/\beta} ] } 
                               \frac{c_{4}}{t^{\alpha / \beta}} \theta^{\omega} (y)  \frac{r^{\alpha - \beta}}{( t^{1/\beta})^{\alpha - \beta}}  \\
                      & \qquad \left( \text{ use $d(x,y) \le d(x,x_0) + d(x_0,y) \le (c_2 + 1) t^{1/\beta}$ for $y \in B(x_0,c_2 t^{1/\beta})$} \right)  \\
                 & \ge \frac{c_{5} (c_{3.1}c_1^{\alpha} - c_{3.2}2^{\alpha}) r^{\alpha - \beta}t }{  t^{\alpha / \beta} }.  
        \end{align*} 
 We thus complete the proof by taking $c_1 = c_{5} (c_{3.1}c_2^{\alpha} - c_{3.2}2^{\alpha})$.  
\end{proof}

\begin{Lemma}  \label{RFGr90} 
  Under  Assumption \ref{Ass10} (1) (2) (3) (4) and Assumption \ref{Ass:Gr}  
  there exist positive constants $c_1, c_2$, $\eta_0$, $T_0$ such that for any $\eta \ge \eta_0$ the following holds; 
       \begin{align*}
               P_x^{\omega} \left( d(x_0, Y_s^{\omega} ) \le 2r \text{ for some $s \in (t, \eta t]$ } \right)  \ge \frac{c_1r^{\alpha - \beta} t}{t^{\alpha / \beta}}
       \end{align*} 
    for almost all $\omega \in \Omega$, all $r \ge 1$, $t \ge T_0$, $x,x_0 \in V(G^{\omega})$ with $d(x,x_0) \le r$, $t \ge r^{\beta}$, 
    $\displaystyle r \ge \max_{z \in B(x_0,c_2 t^{1/\beta})} N_z (\omega)$.   
\end{Lemma}
  
\begin{proof} 
 By Lemma \ref{RFGr70} and Lemma \ref{RFGr80} there exist positive constants $c_1, c_2, c_3, T_0$ such that for almost all $\omega \in \Omega$ 
          \begin{align*}
                \frac{c_1 r^{\alpha - \beta} t}{t^{\alpha / \beta}} 
                       \le    P_x^{\omega} \left( d(x_0,Y_s^{\omega} ) \le 2r \text{ for some $s > t$ } \right)  
                       \le    \frac{c_{2} r^{\alpha - \beta} t}{t^{\alpha / \beta}}  
          \end{align*}
  for $r \ge 1$, $t \ge T_0$, $x,x_0 \in V(G^{\omega})$ with $d(x,x_0) \le r$, $t \ge r^{\beta}$,   
    $\displaystyle r \ge \max_{z \in B(x_0,c_3 t^{1/\beta})} N_z (\omega)$.     
  Take $\eta_0$ such that $\displaystyle c_2 - \frac{c_1}{\eta^{\alpha / \beta -1} } > \frac{c_2}{2} $ for all $\eta \ge \eta_0$.  
   Then we have 
       \begin{align*}
               & P_x^{\omega} \left( d(x_0, Y_s^{\omega} ) \le 2r \text{ for some $s \in (t, \eta t]$ } \right)   \\   
               & \ge  P_x^{\omega} \left( d(x_0, Y_s^{\omega} ) \le 2r \text{ for some $s > t$ } \right) 
                     -  P_x^{\omega} \left( d(x_0, Y_s^{\omega} ) \le 2r \text{ for some $s > \eta t$ } \right)    \\
               & \ge c_2 \frac{r^{\alpha - \beta} t}{t^{\alpha / \beta} } - c_1 \frac{r^{\alpha - \beta} (\eta t)}{ (\eta t)^{\alpha / \beta} } 
                        =  \frac{r^{\alpha - \beta} t}{ t^{\alpha / \beta}} \left( c_2 - \frac{c_1}{ \eta^{\alpha / \beta -1} } \right) .
       \end{align*}
 We  complete the proof by adjusting the constants. 
\end{proof}

\subsection{Consequences of Assumption \ref{Ass30}}
In this subsection, we give easy consequences of  Assumption \ref{Ass30}. 
We use $\varphi (q) = \varphi_C (q) = Cq^{1/\beta} (\log \log q)^{1-1/\beta}$ in this subsection. 
\begin{Lemma}   \label{Tail10}  
 \begin{enumerate}   \renewcommand{\labelenumi}{(\arabic{enumi})} 
      \item  Under Assumption \ref{Ass30} (1), for all $\gamma_1, \gamma_2 >0$, $q > 1$ and for almost all $\omega \in \Omega$ 
                   there exists a positive number $L^{(1)} (\omega ) = L_{x, \epsilon, \gamma_1, \gamma_2, q}^{(1)} (\omega)$ such that 
                                 \begin{align*} 
                                       \gamma_1 q^{n/\beta} \ge \max_{y \in B(x,\gamma_2 q^{n/\beta} )} N_{y} (\omega ), 
                                       \quad \gamma_1 \varphi (q^n) \ge \max_{y \in B(x,\gamma_2  \varphi (q^n)  )} N_{y} (\omega ), 
                                 \end{align*}
                   for all $n \ge L^{(1)} (\omega )$.

      \item Under Assumption \ref{Ass30} (2), for all $\gamma_1, \gamma_2 >0$, $q > 1$ and for almost all $\omega \in \Omega$ 
                   there exists a positive number $L^{(2)} (\omega ) = L_{x, \epsilon, \gamma_1, \gamma_2,q}^{(2)} (\omega)$ such that 
                                 \begin{align*}
                                       \gamma_1 q^{n/\beta} \ge \max_{y \in B(x,\gamma_2 q^n ) } N_{y} (\omega )
                                 \end{align*}
                   for all $n \ge L^{(2)} (\omega )$.

        \item  Set $\psi (t) := t^{1/\beta} h(t)$, where $h(t)$ is non-increasing and $\psi (t)$ is increasing function. 
              Under Assumption \ref{Ass30} (3), for all $\gamma_1, \gamma_2 >0$, $q > 1$ and for almost all $\omega \in \Omega$ 
                   there exists a positive number $L^{(3)} (\omega ) = L_{x, \epsilon, \gamma_1, \gamma_2,q}^{(3)} (\omega)$ such that 
                                 \begin{align*}
                                       \gamma_1 \psi (q^n)  \ge \max_{y \in B(x,\gamma_2  q^{n/\beta} ) } N_{y} (\omega )
                                 \end{align*}
                  for all $n \ge L^{(3)} (\omega )$.   

 \end{enumerate}
\end{Lemma}

\begin{proof}
We can prove (1) (2) (3) similarly, so we prove only the first inequality in (1). 
Since 
             \begin{align*}
                       &  \mathbb{P} \left( \gamma_1 q^{n/\beta} < \max_{y \in B(x, \gamma_2 q^{n/\beta})} N_y  \right)
                                  \le     \sum_{y \in B(x, \gamma_2 q^{n/\beta})}   \mathbb{P} \left( \gamma_1 q^{n/\beta} <  N_y \right)  \\
                       &  \le c (\gamma_2 q^{n/\beta})^{\alpha} f (  \gamma_1 q^{n/\beta} ), 
             \end{align*}
  where we use union bound in the first inequality and use \eqref{NumVer} in the second inequality. 
  The conclusion follows by the Borel-Cantelli lemma. 
\end{proof}

\section{ Proof of Theorem \ref{MainLIL} }  \label{Sec:LIL} 
   
In this section we give the proof of Theorem \ref{MainLIL}. 

\subsection{ Proof of  the LIL}
We follow the strategy as in \cite{Duminil-Copin}. 
  \begin{Theorem}  \label{Upp10}
      Let $\varphi (t) = \varphi_C (t) = C t^{1/\beta}  (\log \log t)^{1-1/\beta}$, where $C> 2^{1+1/\beta} c_{1.2}^{- (\beta -1)/\beta}$.
      Then under Assumption \ref{Ass10} (1) (2) (3) and Assumption \ref{Ass30} (1) the following hold for almost all $\omega \in \Omega$; 
              \begin{align}  
                      & \limsup_{t \to \infty}   \frac{ \sup_{0 \le s \le t} d(Y_0^{\omega} , Y_s^{\omega}) }{\varphi (t) }  \le 1, 
                                 \qquad \text{ $P_x^{\omega}$-a.s. for all $x \in V(G^{\omega})$},    \label{Upp12}   \\  
                      & P_x^{\omega} \left(  \sup_{0 \le s \le t}  d(x,Y_s^{\omega} ) \le \varphi (t)  \text{ for sufficiently large $t$}  \right) = 1, \qquad \text{for all $x \in V(G^{\omega})$}. 
                                   \label{Upp13}
              \end{align}

       In particular, we have  
             \begin{align*}  
                     & \limsup_{t \to \infty}   \frac{d(Y_0^{\omega} , Y_t^{\omega}) }{\varphi (t) }  \le 1, 
                                 \qquad \text{ $P_x^{\omega}$-a.s. for all $x \in V(G^{\omega})$},       \\  
                     & P_x^{\omega} \left( d(x,Y_t^{\omega} ) \le \varphi (t) \text{ for all sufficient large $t$}  \right) = 1, \qquad \text{ for all $x \in V(G^{\omega})$} .  
             \end{align*}     
  \end{Theorem}

\begin{proof} 
Take $\eta >0$ and $\delta \in (0, c_{1.2} \wedge c_{1.4})$ sufficiently small constants 
        which satisfy \\
 $\displaystyle C > 2^{1/\beta} (1+\eta)^{1/\beta} \left( \frac{1}{c_{1.2} - \delta} \right)^{(\beta -1)/\beta}$.  
Set $t_n = (1+\eta )^n$.  

First we estimate $\displaystyle P_x^{\omega} \left( \sup_{0 \le s \le t_{n+1} }d(x, Y_s^{\omega})\ge  2\varphi_C (t_n) \right)$. 
  For all $\delta \in (0,c_{1.2} \wedge c_{1.4})$, using Lemma \ref{RFHK30} we have 
             \begin{align}   \label{Upp16}
                     &  P_x^{\omega} \left( \sup_{0 \le s \le t_{n+1} } d(x,Y_s^{\omega} ) \ge  2\varphi (t_n) \right)  
                             \le  c_1 \exp \left[ -(c_{1.2} - \delta )  \left( \frac{\varphi (t_n) }{ (2 t_{n+1})^{1/\beta} } \right)^{\beta/(\beta -1)} \right]   
                               +   c_2 \exp \left[ -c_3 t_{n+1} \right]   \notag \\
                     & \le  c_1 \exp \left[ -(c_{1.2} - \delta )  \left( \frac{\varphi (t_n) }{ ( 2(1+\eta ) t_{n})^{1/\beta} } \right)^{\beta/(\beta -1)} \right]   
                               +   c_2 \exp \left[ -c_3 t_{n+1} \right]    
             \end{align} 
    for $\displaystyle \sup_{z \in B(x, 2\varphi (t_n) )} N_{z} (\omega )\le \varphi (t_n) \wedge t_{n+1}$. 
 Note that $\displaystyle \sup_{z \in B(x, 2\varphi (t_n) )} N_{z} (\omega )\le \varphi (t_n) \wedge t_{n+1}$ 
        for all $n$ larger than a certain constant $L = L (\omega )$ by Lemma \ref{Tail10} (1). 

 We prove \eqref{Upp12}. Let $\displaystyle C > 2^{1/\beta} (1+\eta)^{1/\beta} \left( \frac{1}{c_{1.2} - \delta} \right)^{(\beta -1)/\beta}$ be as above.
 Since the  last term of \eqref{Upp16} is summable by the definition of $\eta$ and $\delta$.  
 By the Borel-Cantelli lemma we have 
                \begin{align*}
                       \limsup_{n \to \infty}  \frac{ \sup_{0 \le s \le t_{n+1}} d(Y_0^{\omega}, Y_s^{\omega})}{ 2\varphi (t_n) }  \le 1,  
                                 \qquad  \text{ $P_x^{\omega}$-a.s. for all $x \in V(G^{\omega})$}.  
               \end{align*}
   For all $t$ with $t_n \le t < t_{n+1}$ we have 
               \begin{align*}
                      \frac{ \sup_{0 \le s \le t} d(Y_0^{\omega}, Y_s^{\omega})}{ 2\varphi (t) }  
                                 \le  \frac{ \sup_{0 \le s \le t_{n+1}} d(Y_0^{\omega}, Y_s^{\omega})}{ 2\varphi (t_n) } .
               \end{align*}
    Hence we obtain \eqref{Upp12} from the above inequality and adjusting the constants. 

 Next we prove \eqref{Upp13}. 
 Let $\displaystyle C > 2^{1/\beta} (1+\eta)^{1/\beta} \left( \frac{1}{c_{1.2} - \delta} \right)^{(\beta -1)/\beta}$ be as above.  
 Since 
               \begin{align*}
                     P_x^{\omega} \left( \sup_{0 \le s \le t_{n} } d(x,Y_s^{\omega} ) \ge 2\varphi (t_n) \right) 
                             \le P_x^{\omega} \left( \sup_{0 \le s \le t_{n+1} } d(x,Y_s^{\omega} ) \ge 2\varphi (t_n) \right) 
               \end{align*}
  for $t \in [t_n , t_{n+1}]$ and the last term of \eqref{Upp16} is summable by the definition of $\eta$ and $\delta$. By the Borel-Cantelli lemma we have 
             \begin{align}
                         P_x^{\omega} \left(  \sup_{0 \le s \le t} d(Y_0^{\omega} , Y_s^{\omega} ) \le 2 \varphi (t) \text{ for all sufficiently large $t$ } \right) = 1,
                                      \qquad \text{for all $x \in V(G^{\omega})$}.
             \end{align}
      We thus complete the \eqref{Upp13} by adjusting the constants.   
\end{proof}

\begin{Theorem}  \label{Upp20} 
Let $\varphi (t) = \varphi_C (t) = C t^{1/\beta} (\log \log t)^{1- 1/\beta}$, 
  where $\displaystyle 0 < C <\frac{1}{2^{1+1/\beta}} \left( \frac{c_{3.1}}{c_{3.2}} \right)^{1/\alpha}  \left( \frac{1}{c_{2.2}} \right)^{(\beta- 1)/\beta}$.
Then under Assumption \ref{Ass10} (1) (2) (3) and Assumption \ref{Ass30} (1) the following holds;   
          \begin{align*}
                    \limsup_{t \to \infty}  \frac{ d(Y_0^{\omega}, Y_t^{\omega}) }{ \varphi (t) } \ge 1, \qquad \text{$P_x^{\omega}$-a.s. for all $x \in V(G^{\omega})$}.
          \end{align*} 
In particular, we have 
          \begin{align*}
                   &P_x^{\omega} \left(  d(Y_0^{\omega}, Y_t^{\omega})  \ge  \varphi (t) \text{ for sufficiently large $t$} \right) = 1,
                            \qquad \text{for all $x \in V(G^{\omega})$}, \\
                   &  \limsup_{t \to \infty}  \frac{ \sup_{0 \le s \le t}  d(Y_0^{\omega}, Y_s^{\omega}) }{\varphi (t)} \ge 1, 
                               \qquad \text{$P_x^{\omega}$-a.s. for all $x \in V(G^{\omega})$}. 
          \end{align*}
\end{Theorem}  

\begin{proof} 
 Define $\Phi (q) = q^{1/\beta} (\log \log q)^{1-1/\beta}$ and let $C$ be as above. 
Take $\eta >0$ as a sufficiently small constant such that 
          \begin{align*}
                    C<  \frac{1}{2^{1/\beta}} \left\{ \frac{1}{2} \left( \frac{c_{3.1}}{c_{3.2}}  \right)^{1/ \alpha}  - \eta \right\}  \left( \frac{1 }{c_{2.2} } \right)^{(\beta -1)/ \beta}. 
          \end{align*}
 Set $\displaystyle \frac{1}{\lambda} =  \frac{1}{2} \left( \frac{c_{3.1}}{c_{3.2}} \right)^{1/\alpha} - \eta$. 
 Note that $\displaystyle c_{3.1} \lambda^{\alpha} -c_{3.2} 2^{\alpha}>0$ and $\displaystyle c_{2.2} ( 2^{1/\beta} C\lambda)^{\beta/(\beta -1)} < 1$.

We prove that 
         \begin{align}  \label{Upp22}
                   \sum_{n}  P_x^{\omega} \left( A_n^{\omega} \mid \mathcal{F}_{2^n}^{\omega} \right)  = \infty,
         \end{align}
 where $A_n^{\omega} = \left\{ d(Y_{2^n}^{\omega}, Y_{2^{n+1}}^{\omega} ) \ge 2 \varphi (2^{n+1}) \right\}$ and $\mathcal{F}_t^{\omega} = \sigma \left( Y_s^{\omega} \mid s \le t \right)$. 
 To prove \eqref{Upp22}, first note that by Theorem \ref{Upp10}  there exists a sufficiently large constant $C_1$ such that for almost all $\omega \in \Omega$ 
         \begin{align*}
                   d(x,Y_{2^n}^{\omega}) \le C_1 \Phi (2^n)   \qquad \text{for sufficiently large $n$ (say $n \ge \tilde{N}_1$), \qquad $P_x^{\omega}$-a.s.}
         \end{align*}            
 Set $B_n^{\omega} = A_n^{\omega} \cap \{ d(Y_0^{\omega} , Y_{2^n}^{\omega}) \le C_1 \Phi (2^n) \}$. 
 Then we have 
           \begin{align}  \label{Upp24}
                  &  P_x^{\omega} \left( A_n^{\omega} \mid \mathcal{F}_{2^n}^{\omega} \right)  \ge  P_x^{\omega} \left( B_n^{\omega} \mid \mathcal{F}_{2^n}^{\omega} \right)  
                      =  1_{ \{ d( Y_0^{\omega} , Y_{2^n}^{\omega} ) \le C_1 \Phi (2^{n}) \} } 
                                        P_{Y_{2^n}^{\omega} }^{\omega} \left( d (Y_0^{\omega} , Y_{2^{n+1} -2^n}^{\omega}) \ge 2\varphi (2^{n+1}) \right)   \notag \\
                  & \ge  \left( \inf_{u \in B^{\omega} (x,C_1 \Phi (2^{n}))} P_u^{\omega} \left( d(Y_0^{\omega}, Y_{2^n}^{\omega}) \ge 2 \varphi (2^{n+1}) \right) \right)  
                               \cdot 1_{ \{ d( Y_0^{\omega} , Y_{2^n}^{\omega} ) \le C_1 \Phi (2^{n}) \} } , \qquad \text{$P_x^{\omega}$-a.s.} 
          \end{align} 
We consider the first term of \eqref{Upp24}. Take $u \in B^{\omega} (x,C_1 \Phi (2^{n}))$. 
Since $1+\epsilon < \beta$, there exists a positive integer $\tilde{N}_2 =\tilde{N}_2 (\lambda)$ (which does not depend on $u, \omega$) such that $d(u,v)^{1+\epsilon} \le 2^n$ for all $n \ge \tilde{N}_2 $ and 
  $v \in B^{\omega} (u,\lambda \varphi (2^{n+1}))$. 
So for all $n\ge \tilde{N}_2$ with $2^n \wedge 2\varphi (2^{n+1}) \ge N_{u} (\omega )$, using \eqref{LHK} and \eqref{Vol} we have 
          \begin{align}
                  &P_u^{\omega} \left( d(Y_0^{\omega}, Y_{2^n}^{\omega}) \ge 2\varphi (2^{n+1}) \right)
                            \ge  P_u^{\omega} \left( 2\varphi (2^n)  \le d(Y_0^{\omega}, Y_{2^n}^{\omega}) \le \lambda \varphi (2^{n+1}) \right)   \notag \\
                  &  = \sum_{ \substack{v \in V(G^{\omega}) \\  \varphi (2^{n+1}) \le d(u,v) \le \lambda \varphi (2^{n+1})} } q_{2^n}^{\omega} (u,v) \theta^{\omega} (v)  \notag \\
                  & \ge \sum_{ \substack{v \in V(G^{\omega}) \\  2\varphi (2^{n+1}) \le d(u,v) \le \lambda \varphi (2^{n+1})} } 
                             \frac{c_{2.1}}{(2^n)^{\alpha/\beta}} \exp \left[ - c_{2.2} \left( \frac{d(u,v)}{(2^n)^{1/\beta}} \right) ^{\beta/(\beta -1)} \right] \theta^{\omega} (v) \notag \\ 
                  & \ge \frac{c_{2.1}}{(2^n)^{\alpha / \beta}} \exp \left[ -c_{2.2} \left( \frac{\lambda \varphi (2^{n+1})}{(2^n)^{1/\beta}} \right)^{\beta / (\beta -1)} \right]
                              \theta^{\omega} (\{  v \in V(G^{\omega}) \mid 2\varphi (2^{n+1}) \le d(u,v) \le \lambda \varphi (2^{n+1}) \} )   \notag \\
                 & \ge c_{2.1} (c_{3.1} \lambda^{\alpha} -c_{3.2}2^{\alpha}) C^{\alpha} \left( \frac{1}{(n+1) \log 2} \right)^{c_{2.2} (2^{1/\beta}\lambda C)^{\beta / \beta -1} }
                        \left( \log \log 2^{n+1} \right)^{(\beta-1)\alpha/\beta} .     \notag              
         \end{align} 
By the above estimate we have 
         \begin{align}  \label{Upp26} 
                & \inf_{u \in B^{\omega} (x,C_1 \Phi (2^{n}))} P_u^{\omega} \left( d(Y_0^{\omega}, Y_{2^n}^{\omega})  \ge 2\varphi (2^{n+1}) \right)  \notag  \\
                &  \ge c_{2.1} (c_{3.1} \lambda^{\alpha} -c_{3.2}2^{\alpha}) C^{\alpha} \left( \frac{1}{(n+1) \log 2} \right)^{c_{2.2} (2^{1/\beta}\lambda C)^{\beta / \beta -1} }
                        \left( \log \log 2^{n+1} \right)^{(\beta-1)\alpha/\beta}  
        \end{align}
 for $n \ge \tilde{N}_2$ with $\displaystyle \max_{u \in B(x,C_1 \Phi (2^{n}))} N_{u} (\omega ) \le 2^n \wedge 2\varphi (2^{n+1})$.  
 By Lemma \ref{Tail10} (1),  $\displaystyle \max_{u \in B(x,C_1 \Phi (2^{n}))} N_{u} (\omega ) \le 2^n \wedge 2\varphi (2^{n+1})$ 
      holds for sufficiently large $n$ (say $n \ge \tilde{N}_3 = \tilde{N}_3 (\omega)$).  
 Hence by \eqref{Upp24} and \eqref{Upp26} we have 
          \begin{align}  \label{Upp25}
                  P_x^{\omega} \left( A_n^{\omega} \mid \mathcal{F}_{2^n}^{\omega} \right)  
                    \ge    c_{2.1} (c_{3.1} \lambda^{\alpha} -c_{3.2}2^{\alpha}) C^{\alpha} \left( \frac{1}{(n+1) \log 2} \right)^{c_{2.2} (2^{1/\beta} \lambda C)^{\beta / \beta -1} }
                                 \left( \log \log 2^{n+1} \right)^{(\beta-1)\alpha/\beta}  
         \end{align}
 for $n \ge \tilde{N}_1 \vee \tilde{N}_2 \vee \tilde{N}_3$. We thus complete to show \eqref{Upp22}.

By  \eqref{Upp22} and the second Borel-Cantelli lemma, 
  $d(x,Y_{2^n}^{\omega} ) \ge \varphi (2^n)$ or $d(x,Y_{2^{n+1}}^{\omega}) \ge  \varphi (2^n)$ for infinitely many $n$.
Hence 
       \begin{align*}
               \limsup_{ t \to \infty}  \frac{d(Y_0^{\omega},Y_t^{\omega})}{\varphi (t)}  \ge  1.
       \end{align*}
 We thus complete the proof. 
\end{proof}

By Theorem \ref{Upp10}, \ref{Upp20} and \ref{RF0-1-30} we obtain \eqref{MainLIL2}.

\subsection{Another law of the iterated logarithm}  \label{Subsec:AnotherLIL} 

The proof of Theorem \ref{MainLIL} (2) is quite similar to that of \cite[Theorem 4.1]{KN} 
 by using Lemma \ref{RFHK30}, Corollary \ref{RFHK90}, Lemma \ref{RFHK70},  Theorem \ref{RF0-1-30} and Lemma \ref{Tail10} (2).    
So we omit the proof.

\section{Lower Rate Function}  \label{Sec:LRF}
In this section we give the proof of Theorem \ref{Thm:LRF}.
We follow the strategy as in \cite[Section 4.1]{SW}.

 \begin{Theorem}  \label{LRF10} 
      Suppose that Assumption \ref{Ass10} (1) (2) (3) (4). In addition suppose that there exists a positive constant $c$ such that 
                 $\theta^{\omega} (x) \ge c$ for all $x \in V(G^{\omega})$ in the case of CSRW.  
      Let $\alpha/\beta >1$, $h:[0,\infty) \rightarrow (0,\infty)$ be a function such that $h(t) \searrow 0$ as $t \to \infty$, 
          $\varphi (t) := t^{1/\beta} h(t)$ be increasing for all sufficiently large $t$ and satisfy Assumption \ref{Ass30} (3).  
      If the function $h(t)$ satisfies
                   \begin{align}  \label{LRF11}
                             \int_1^{\infty} \frac{1}{t} h(t)^{\alpha - \beta} dt < \infty
                   \end{align}
        then  for almost all $\omega \in \Omega$ and all $x \in V(G^{\omega})$ we have 
                   \begin{align*}
                          P_x^{\omega} \left( d(x,Y_t^{\omega}) \ge t^{1/\beta} h(t)  \text{ for all sufficiently large $t$}\right) = 1. 
                   \end{align*}
   \end{Theorem}

\begin{proof}
Set $\varphi (t) := t^{1/\beta} h(t)$, $t_n := 2^n$ and 
    $A_n^{\omega} := \{ d(x,Y_s^{\omega}) \le \varphi (s) \text{ for some $s \in (t_{n},t_{n+1}]$ } \}$.  
Note that there exists a constant $c_1$ such that $\varphi (s) \le 2c_1 \varphi (t_n)$ for all sufficiently large $n$ (say $n \ge N_1$) and for all $s \in (t_n, t_{n+1}]$.   
Then by Lemma \ref{RFGr70} we have 
                \begin{align*}
                          & P_x^{\omega} \left( A_n^{\omega} \right) 
                                    \le P_x^{\omega} \left( d(x,Y_s^{\omega} ) \le 2c_1 \varphi (t_{n}) \text{ for some $s>t_n$}  \right)     
                                    \le \frac{c_2\varphi (t_n)^{\alpha - \beta} t_n}{t_n^{\alpha / \beta }} 
                \end{align*}  
  for $n$ with 
                \begin{align}  \label{LRF12} 
                    \begin{split}
                               &n \ge N_1,    \quad  2^n \ge T_0, \text{where $T_0$ is as in Lemma \ref{RFGr70}},  \quad  t_n^{1/\beta} \ge c_1 \varphi (t_n),  \\ 
                               &c_1 \varphi (t_n) \ge  \max_{z \in B(x,c_1 \varphi (t_n))} N_{z} (\omega).  
                    \end{split}
                 \end{align}  
  Note that \eqref{LRF12} is satisfied for sufficiently large $n$ (say $n \ge N_2 = N_2 (\omega))$ 
     by Assumption \ref{Ass30} (3) and Lemma \ref{Tail10} (3). 
  Thus 
                 \begin{align*}
                      &\sum_{n \ge N_2 (\omega)} P_x^{\omega} (A_n^{\omega})  
                                  \le   \sum_{n \ge N_2 (\omega)} \frac{c_2\varphi (t_n)^{\alpha - \beta} t_n}{t_n^{\alpha / \beta }} 
                                  =       \sum_{n \ge N_2 (\omega)} \frac{c_2 h(t_n)^{\alpha - \beta} t_n}{ t_n }   \\
                      &   \le    \sum_{n \ge N_2 (\omega)} \frac{c_3 h(t_n)^{\alpha - \beta} (t_n - t_{n-1})}{ t_n }   
                             \le  c_4 \int_{t_{N_2 -1}}^{\infty} \frac{h(s)^{\alpha - \beta}}{ s} ds .
                 \end{align*} 
 Since the above is integrable by \eqref{LRF11}, by the Borel-Cantelli lemma we have 
                     \begin{align*}
                          P_x^{\omega} \left( d(x,Y_t^{\omega}) \ge t^{1/\beta} h(t) \text{ for all sufficiently large $t$} \right) = 1. 
                   \end{align*}
 We thus complete the proof.              
\end{proof}

\begin{Theorem} \label{LRF30} 
      Suppose that Assumption \ref{Ass10} (1) (2) (3) (4) hold. 
       In addition suppose that there exists a positive constant $c$ such that 
                 $\theta^{\omega} (x) \ge c$ for all $x \in V(G^{\omega})$ in the case of CSRW.     
      Let $\alpha/\beta >1$, $h:[0,\infty) \rightarrow (0,\infty)$ be a function such that $h(t) \searrow 0$ as $t \to \infty$, 
            $\varphi (t) := t^{1/\beta} h(t)$ be increasing for all sufficiently large $t$ and  satisfy Assumption \ref{Ass30} (3).
      If the function $h(t)$ satisfies
                   \begin{align}  \label{LRF32}
                             \int_1^{\infty} \frac{1}{t}  h(t)^{\alpha - \beta} dt  = \infty
                   \end{align}
        then  for almost all $\omega \in \Omega$ and all $x \in V(G^{\omega})$ 
                   \begin{align} \label{LRF31}
                          P_x^{\omega} \left( d(x,Y_t^{\omega}) \ge \varphi (t)  \text{ for all sufficiently large $t$} \right) = 0.
                   \end{align}
\end{Theorem}

We cite the following form of  the Borel-Cantelli Lemma (see \cite[Lemma 4.15]{SW}, \cite[Lemma B]{Takeuchi}, \cite[Theorem 1]{CE}). 
\begin{Lemma}  \label{LRF20}
  Let $\{ A_k \}_{k \ge 1}$ be a family of event which satisfies the following conditions;
          \begin{enumerate}   \renewcommand{\labelenumi}{(\arabic{enumi})} 
                  \item $\displaystyle \sum_k P(A_k)  = \infty$,  
                  \item $\displaystyle   P(\limsup A_k)  = 0 \text{ or } 1$, 
                  \item There exist two constants $c_1,c_2$ such that for each $A_j$  there exist $A_{j_1}, \cdots, A_{j_s} \in \{ A_k \}_{k \ge 1}$ such that 
                          \begin{enumerate} \renewcommand{\labelenumi}{(\roman{enumi})}  
                                    \item  $\displaystyle \sum_{i=1}^s P (A_j \cap A_{j_i} ) \le c_1 P(A_j)$,
 
                                    \item  for any $k \in \{ j+1, j+2, \cdots \} \setminus \{ j_1, j_2, \cdots, j_s \} $ we have $\displaystyle P(A_j \cap A_k) \le c_2 P(A_j) P(A_k)$.
                          \end{enumerate}
          \end{enumerate}
 Then infinitely many events $\{ A_k \}_{k \ge 1}$ occur with probability $1$. 
\end{Lemma}

\begin{proof}[Proof of Theorem \ref{LRF30}.]  
First we prepare preliminary facts. 
 Since $h(t) \searrow 0$ as $ t \to \infty$, there exists a positive constant $T_1$ such that $h(t) < 1$ for all $t \ge T_1$. 
    So there exists a constant $\kappa \in (0,1)$ such that $\varphi (t) \le (\kappa t)^{1/\beta}$ for $t \ge T_1$.  
    Take $\eta >1 \vee \eta_0$ (where $\eta_0$ is as in Lemma \ref{RFGr90}) with $1 - \frac{1}{\eta} \ge \kappa$ 
       and $c_1 = c_1 (\eta) \in (0,1)$ such that $2c_1 (\eta^{n+1})^{1/\beta} \le (\eta^n)^{1/\beta}$ for all $n$. 
    Note that for all $s $ with $\eta^{n+1} \le s \le \eta^{n+2}$ we have 
               \begin{align}  \label{LRF32}
                        \varphi (\eta^{n+1}) = (\eta^{n+1})^{1/\beta} h(\eta^{n+1}) \ge 2c_1 (\eta^{n+2})^{1/\beta} h(s) \ge 2c_1 \varphi (s) ,  
               \end{align}
      and for all sufficiently large $i,j$ with $i \ge j+2$ and $\eta^j \ge T_1$ (say $j \ge N_1$) we have
               \begin{align}  \label{LRF33}
                        (c_1 \varphi (\eta^{i+1}) )^{\beta}  \stackrel{\eqref{LRF32}}{\le } \varphi (\eta^{i})^{\beta} \le \kappa \eta^i
                         \stackrel{1-\frac{1}{\eta} \ge \kappa } \le \eta^{i} - \eta^{i-1} \le \eta^{i} - \eta^{j+1}.
               \end{align}

 Now we prove \eqref{LRF31}.             
 Set $A_n^{\omega} := \{ d(Y_0^{\omega}, Y_s^{\omega} ) \le 2c_{1} \varphi (\eta^{n+1}) \text{ for some $s \in (\eta^n, \eta^{n+1} ]$ } \}$. 
 We use Lemma \ref{LRF20} to show that infinitely many $A_n^{\omega}$ occur with probability $1$. 
 
 Note that $\eta^{n} \ge (c_{1} \varphi (\eta^{n+1}) )^{\beta}$ for sufficiently large $n$ (say $n \ge N_{2} = N_{2} (\eta)$) by \eqref{LRF33}. 
 By Lemma \ref{RFGr90}  we have 
                   \begin{align*}
                         P_x^{\omega} \left( A_n^{\omega} \right) 
                                 \ge c_{2} \frac{ (c_1 \varphi (\eta^{n+1}))^{\alpha - \beta} \eta^n}{ \eta^{n\alpha / \beta}}
                   \end{align*}
    for $\eta \ge \eta_0$ (where $\eta_0$ is as in Lemma \ref{RFGr90}) and $n \ge N_2$ with 
                    \begin{align}  \label{LRF40} 
                           \eta^{n} \ge T_0, \text{ where $T_0$ is as in Lemma \ref{RFGr90}}, 
                           \qquad  c_1 \varphi (\eta^{n+1}) \ge  \max_{z \in B(x,c_2 \eta^{n/\beta})} N_{z} (\omega). 
                     \end{align}    
 Note that \eqref{LRF40} holds for sufficiently large $n$ (say $n \ge N_{3} (\omega)$)  by Assumption \ref{Ass30} (3) and Lemma \ref{Tail10} (3).    
    Hence 
                    \begin{align*}
                            & \sum_{n \ge N_{3}} P_x^{\omega} \left( A_n^{\omega} \right) 
                                    \ge \sum_{n \ge N_{3}}  \frac{c_2 (c_1 \varphi (\eta^{n+1}) )^{\alpha - \beta} \eta^n}{ \eta^{n\alpha / \beta}}  
                                    =  \sum_{n \ge N_{3}}  c_{2} c_{1}^{\alpha - \beta} \eta^{\alpha / \beta}  
                                         \frac{ h(\eta^{n+1})^{\alpha - \beta} }{ \eta \cdot \eta^{n+1} } \eta^{n+1}    \\
                            &  =  \sum_{n \ge N_{3}}   \frac{ c_{2} c_{1}^{\alpha - \beta} \eta^{\alpha / \beta} }{ \eta \cdot (\eta -1)}  
                                              \frac{ h (\eta^{n+1})^{\alpha - \beta} }{ \eta^{n+1} } (\eta^{n+2} - \eta^{n+1})  
                                    \ge   \frac{ c_{2} c_{1}^{\alpha - \beta} \eta^{\alpha / \beta} }{\eta (\eta -1)} 
                                                 \int_{\eta^{N_3+1}}^{\infty}  \frac{h(s)^{\alpha - \beta}}{s} ds  .
                    \end{align*}   
 Thus we have $\displaystyle \sum_{n } P_x^{\omega} \left( A_n^{\omega} \right) = \infty$ by \eqref{LRF32}. 
 
The condition $(2)$ in Lemma \ref{LRF20} is immediate from Theorem \ref{RF0-1-30}, 
  since $\limsup_k A_k^{\omega}$ is a tail event.   

Next we show the condition $(3)$ in Lemma \ref{LRF20}. 
Set $\sigma_n^{\omega} := \inf \{ t \in (\eta^n, \eta^{n+1}] \mid d(Y_0^{\omega}, Y_t^{\omega}) \le 2 c_1 \varphi (\eta^{n+1}) \}$. 
Then for $i \ge j+2$  we have 
         \begin{align}  \label{LRF34}
                   & P_x^{\omega} (A_i^{\omega} \cap A_j^{\omega}) 
                           = P_x^{\omega} (\sigma_j \le \eta^{j+1}, \sigma_i \le  \eta^{i+1})  \notag  \\
                   &  =E_x^{\omega} \left[ 1_{ \{ \sigma_j \le \eta^{j+1} \} }  
                              P_{Y_{\sigma_j}}^{\omega}  \left(  d(x,Y_t^{\omega}) \le 2c_1 \varphi (\eta^{i+1} ) 
                                    \text{ for some $t \in (\eta^{i} - \sigma_j, \eta^{i+1} - \sigma_j]$} \right)  \right]   \notag   \\
                   &  \le   E_x^{\omega} \left[ 1_{ \{ \sigma_j \le \eta^{j+1} \} }    
                              P_{Y_{\sigma_j}}^{\omega}  \left(   d(x,Y_t^{\omega}) \le 2c_1 \varphi (\eta^{i+1} ) 
                                    \text{ for some $t > \eta^{i} - \eta^{j+1}$} \right)  \right]    \notag  \\
                   &  \le  \left( \sup_{z:d(x,z) \le 2c_1 \varphi (\eta^{j+1} )}   P_z^{\omega} \left( d(x,Y_t^{\omega}) \le 2c_1 \varphi (\eta^{i+1}) 
                                     \text{ for some $t > \eta^{i} - \eta^{j+1}$} \right)  \right) \cdot 
                                     P_x^{\omega} \left( \sigma_j \le \eta^{j+1} \right) .  
        \end{align}         
  By Lemma \ref{RFGr70}, for any $i \ge j+2$ with 
               \begin{align}  \label{LRF35}
                       \eta^{i} - \eta^{j+1} \ge (c_1 \varphi (\eta^{i+1}))^{\beta},  
                             \quad 2 c_1 \varphi (\eta^{j+1}) \le c_1 \varphi (\eta^{i+1}), 
                             \quad  \varphi (\eta^{i+1}) \ge \max_{z \in B(x,  \varphi (\eta^{i+1}) )} N_z (\omega) 
               \end{align}
       we have 
               \begin{align}  \label{LRF36}
                    & \left( \sup_{z:d(x,z) \le 2c_1 \varphi (\eta^{j+1} )} 
                                P_z^{\omega} \left( d(x,Y_t^{\omega}) \le 2c_1 \varphi (\eta^{i+1}) 
                                     \text{ for some $t > \eta^{i} - \eta^{j+1}$} \right)  \right)    \notag  \\   
                    & \le  \frac{c_{3} \left( c_1 \varphi (\eta^{i+1} )  \right)^{\alpha - \beta}  (\eta^i - \eta^{j+1}) }{ (\eta^i - \eta^{j+1})^{\alpha / \beta} }  
                          \le   \frac{c_4 \left( c_1 \varphi (\eta^{i+1} )  \right)^{\alpha - \beta} \eta^i }{ (\eta^{i} )^{\alpha / \beta} }  .
               \end{align}
     \eqref{LRF35} holds for sufficiently large $i,j$ with $i \ge j+2$ (say $j \ge N_4=N_4 (\omega)$) by
     \eqref{LRF32}, \eqref{LRF33}, Assumption \ref{Ass30} (3) and Lemma \ref{Tail10} (3).  
   By Lemma \ref{RFGr90}, for any $i$ with  
                 \begin{align} \label{LRF37} 
                          \eta^{i} \ge T_0, \text{ where $T_0$ is as in Lemma \ref{RFGr90}}, \quad   \eta^{i} \ge (c_1 \varphi (\eta^{i+1}) )^{\beta},  
                                \quad c_1 \varphi (\eta^{i+1}) \ge  \max_{v \in B(x, c_5 \eta^{i/\beta})} N_{v} (\omega)
                 \end{align}
         we have 
                \begin{align}  \label{LRF38}
                      & \frac{ \left( c_1 \varphi (\eta^{i+1} )  \right)^{\alpha - \beta} \eta^i }{ (\eta^{i} )^{\alpha / \beta} }   
                           \le c_{6} P_x^{\omega} \left( d(x,Y_t^{\omega}) \le 2c_1 \varphi (\eta^{i+1}) 
                                     \text{ for some $t \in ( \eta^{i}, \eta^{i+1}]$} \right)   \notag \\
                      & = c_{6} P_x^{\omega} \left( A_i^{\omega} \right). 
                \end{align}   
     \eqref{LRF37} holds for sufficiently large $j$ (say $j \ge N_5=N_5 (\omega)$) by \eqref{LRF32}, Assumption \ref{Ass30} (3) and Lemma \ref{Tail10} (3).  
   Hence by \eqref{LRF34}, \eqref{LRF36} and \eqref{LRF38}
         we have $P_x^{\omega} \left( A_i^{\omega} \cap A_j^{\omega} \right) \le c P_x^{\omega} (A_i^{\omega})  P_x^{\omega} (A_j^{\omega})$ 
        for sufficiently large $j$ ($j \ge N_6 := N_4 \vee N_5)$ and $i \ge j+2$.  
    In the case of $i =j+1$ we have  $P_x^{\omega} \left( A_{j+1}^{\omega} \cap A_j^{\omega} \right) \le P_x^{\omega}  (A_{j}^{\omega})$. 
   Thus we obtain the condition $(3)$ of Lemma \ref{LRF20} for $\{ A_i^{\omega} \}_{i \ge N_6}$.    

By Lemma \ref{LRF20},  we thus complete the proof.
\end{proof}

By Theorem \ref{LRF10} and Theorem \ref{LRF30} we complete the proof of Theorem \ref{Thm:LRF}.

\section{Ergodic media} \label{Sec:Erg}

In this section, we consider the case $G=(V,E) = \mathbb{Z}^d$ and obtain 
Theorem \ref{Thm:Const} under Assumption \ref{Ass50}.  
We follow the strategy as in \cite{Duminil-Copin}

\subsection{Ergodicity of the shift operator on $\Omega^{\mathbb{Z}}$}
We consider Markov chains on the random environment, which is called the environment seen from the particle,  
according to Kipnis and Varadhan \cite{KV}.

Let $\Omega = [0,\infty )^{E }$ and define  $\mathscr{B}$ as the natural $\sigma$-algebra (generated by coordinate maps).  
We write  $\mathcal{Y} = \Omega^{\mathbb{Z}}$, $\mathscr{Y} = \mathscr{B}^{\otimes \mathbb{Z}}$. 
If each conductance may take the value $0$, we regard $0$ as the base point and define 
$\mathcal{C}_0 (\omega) = \{ x \in \mathbb{Z}^d \mid 0 \overset{\omega}{\longleftrightarrow} x \} = V(G^{\omega})$, 
 where  $0 \overset{\omega}{\longleftrightarrow} x $ means that there exists a path $\gamma = e_1 e_2 \cdots e_k$ from $0$ to $x$ such that $\omega (e_i) > 0$ for all $i=1, 2,\cdots, k$. 
Define $\Omega_0 = \{ \omega \in \Omega \mid \sharp \mathcal{C}_0 (\omega) = \infty \}$ 
 and $\mathbb{P}_0 = \mathbb{P} (\cdot \mid \Omega_0)$.

Next we consider the Markov chains seen  from the particle.
Recall that $\{ X_n^{\omega} \}_{n \ge 0}$ is the discrete time random walk which is introduced in Section \ref{Subsec:FW}. 
Let $\omega_n ( \cdot ) = \omega ( \cdot + X_n^{\omega} ) = \tau_{X_n^{\omega} } \omega (\cdot ) \in \Omega$.  
We can regard this Markov chain $\{ \omega_n \}_{n \ge 0} $ as being defined on $\mathcal{Y}  = \Omega^{\mathbb{Z}}$.  
We define a probability kernel $Q: \Omega_0 \times \mathscr{B} \to [0,1]$ as 
           \begin{align*}
                       Q(\omega , A) =    \frac{1}{ \sum_{e^{\prime} : |e^{\prime} | =1}   \omega_{e^{\prime}} }   
                                  \sum_{v: |v| = 1}  \omega_{0v}   1_{ \{ \tau_v \omega \in A \} }  .
           \end{align*}
This is nothing but the transition probability of the Markov chain $\{ \omega_n \}_{n \ge 0}$. 

Next we define the probability measure on $(\mathcal{Y}, \mathscr{Y})$ as 
           \begin{align*}
                       \mu \left(   (\omega_{-n} , \cdots, \omega_n) \in B   \right)   
                                   =  \int_B \mathbb{P}_0 (d\omega_{-n} ) Q(\omega_{-n} ,d\omega_{-n+1} )  \cdots Q(\omega_{n-1}, d\omega_n ) .
          \end{align*}
 By the above definition, $\{ \tau_{X_k^{\omega} } \omega \}_{k \ge 0}$ has the same law in 
$\mathbb{E}_0 ( P_0^{\omega} (\cdot ) )$ as $(\omega_0, \omega_1, \cdots)$ has in $\mu$, 
 that is, 
           \begin{align}  \label{SameDist}
                      \mathbb{E}_0 \left[ P_0^{\omega} (\{ \tau_{X_k^{\omega} } \omega \}_{k \ge 0} \in B ) \right]   =  \mu (  (\omega_0, \omega_1, \cdots) \in B )
           \end{align}
  for any $B \in \mathscr{Y}$.

We need the following Theorem. 
Let $T : \mathcal{Y} \to \mathcal{Y}$ be a shift operator of $\mathcal{Y}$, that is, 
         \begin{align*}
                          (T \omega)_n =  \omega_{n+1}.
         \end{align*}

\begin{Theorem}  \label{Thm:Erg}
Under Assumption \ref{Ass50}, $T$ is ergodic with respect to $\mu$.
\end{Theorem}
The proof is similar to \cite[Proposition 3.5]{BB}, so we omit it.

\subsection{The Zero-One law} 
The purpose of this subsection is to give the Zero-One law (see Proposition \ref{Erg0-1}). 
    Let $a\ge 0$ and $A_1^{\omega} (a) , A_2^{\omega} (a), A_3^{\omega} (a) $ be the events 
               \begin{align*}
                         A_1^{\omega} (a)  &= \left\{ \limsup_{n\to \infty}  \frac{ d(X_0^{\omega}, X_n^{\omega} ) }{ n^{1/\beta}  (\log \log n)^{1-1/\beta} } > a \right\} ,   \\
                         A_2^{\omega} (a)  &= \left\{ \limsup_{n\to \infty}  \frac{ \sup_{0 \le k \le n} d(X_0^{\omega}, X_k^{\omega} ) }{ n^{1/\beta}  (\log \log n)^{1-1/\beta} } > a \right\} ,   \\   
                         A_3^{\omega} (a)  &= \left\{ \liminf_{n \to \infty}   \frac{ \max_{0 \le k \le n} d(X_0^{\omega}, X_k^{\omega} ) }{ n^{1/\beta}  (\log \log n)^{-1/\beta} } > a \right\}. 
               \end{align*}
Define  
                     \begin{align*}
                               \tilde{A}_i (a) = \left\{ \omega \in \Omega \mid \text{  $A_i^{\omega} (a)$ holds for  $P_x^{\omega}$-a.s. and for all $x \in \mathcal{C}_0 (\omega)$} \right\} .
                      \end{align*}

\begin{Proposition}  \label{Erg0-1}
 $\mathbb{P}_0( \tilde{A}_i (a))$ is either $0$ or $1$.
\end{Proposition}

\begin{proof} 
See \cite[Proposition 5.2]{KN}. 
\end{proof}

\subsection{Proof of Theorem \ref{Thm:Const}}
In this subsection we discuss the proof of Theorem \ref{Thm:Const}. 
Recall $T_0^{\omega} =0$, $T_{n+1}^{\omega} = \inf \{ t>T_n^{\omega} \mid Y_t^{\omega} \neq Y_{T_n^{\omega}}^{\omega} \}$ and $X_n^{\omega} = Y_{T_n^{\omega}}^{\omega}$.

First we consider the CSRW.  
$\{ T_{n+1}^{\omega} - T_n^{\omega} \}_{n \ge 0}$ is a family of  i.i.d. random variables whose distributions are exponential with mean $1$, 
so the law of large number gives us 
             \begin{align*}
                    \frac{T_n^{\omega}}{n}  \to 1   \qquad   \text{$P_0^{\omega}$-a.s.} 
            \end{align*} 
Thus
    \begin{align*}
             \limsup_{t \to \infty}  \frac{d(Y_0^{\omega}, Y_t^{\omega})}{ t^{1/\beta} (\log \log t)^{1-1/\beta}} 
                        &=  \limsup_{n \to \infty}  \frac{d(X_0^{\omega}, X_n^{\omega})}{n^{1/\beta} (\log \log n)^{1-1/\beta}} , \\ 
             \limsup_{t \to \infty}  \frac{ \sup_{0\le s \le t} d(Y_0^{\omega}, Y_s^{\omega})}{ t^{1/\beta} (\log \log t)^{1-1/\beta}} 
                        &=  \limsup_{n \to \infty}  \frac{\sup_{0\le k \le n}d(X_0^{\omega}, X_k^{\omega})}{n^{1/\beta} (\log \log n)^{1-1/\beta}} , \\
             \liminf_{ t \to \infty}   \frac{ \sup_{0 \le s \le t} d(Y_0^{\omega}, Y_s^{\omega}) }{t^{1/\beta} (\log \log t)^{-1/\beta}} 
                        &=   \liminf_{ n \to \infty}   \frac{ \sup_{0 \le k \le n} d(X_0^{\omega}, X_k^{\omega}) }{n^{1/\beta} (\log \log n)^{-1/\beta}}.  
    \end{align*}
By Assumption \ref{Ass50}, Proposition \label{0-1-2} and Theorem \ref{MainLIL} we obtain Theorem \ref{Thm:Const}.

Next we consider the VSRW. $\{ T_{n+1}^{\omega} - T_n^{\omega} \}_{n \ge 0}$ are non-i.i.d., and the distribution of $T_{n+1}^{\omega} - T_n^{\omega}$ is exponential with 
  mean $\displaystyle \frac{1}{ \pi^{\omega} (X_n^{\omega}) }$. 
  Write $S_x^{\omega}$ be a exponential random variable with parameter $\pi^{\omega} (x)$ and  $ \bar{S}_x (\bar{\omega}) := S_x^{\bar{\omega}_0} $, 
   $(\bar{\omega} \in \mathcal{Y})$. 
    Then by \eqref{SameDist} and the ergodicity we have 
             \begin{align*}
                    &\frac{1}{n} T_n^{\omega} = \frac{1}{n} \sum_{k=0}^{n-1} S_{X_k^{\omega}}^{\omega} 
                                       \stackrel{d}{=} \frac{1}{n} \sum_{k=0}^{n-1} \bar{S}_{0} (T^k \bar{\omega})   \to  \mathbb{E}^{\mu} \left[ \bar{S}_0 \right] \\
                   &= \mathbb{E} \left[ E_0^{\omega} [ S_0^{\omega} ] \right]  
                                            =   \int_{\Omega} \int_0^{\infty}  x \pi^{\omega} (0) \exp (- \pi^{\omega} (0) x ) dx d\mathbb{P} 
                                            = \mathbb{E} \left[ \frac{1}{\pi^{\omega} (0)} \right] . 
             \end{align*}
Thus
     \begin{align*}
               \limsup_{t \to \infty}  \frac{d(Y_0^{\omega}, Y_t^{\omega})}{t^{1/\beta} (\log \log t)^{1-1/\beta} }  
                       &=  \left( \frac{1}{   \mathbb{E} \left[ \frac{1}{\pi^{\omega} (0)} \right]  } \right)^{1/\beta} 
                             \limsup_{n \to \infty}  \frac{d(X_0^{\omega}, X_n^{\omega})}{ n^{1/\beta} (\log \log n)^{1-1/\beta}},   \\ 
               \limsup_{t \to \infty}  \frac{ \sup_{0\le s \le t} d(Y_0^{\omega}, Y_s^{\omega})}{ t^{1/\beta} (\log \log t)^{1-1/\beta}} 
                        &=  \left( \frac{1}{   \mathbb{E} \left[ \frac{1}{\pi^{\omega} (0)} \right]  } \right)^{1/\beta}  
                              \limsup_{n \to \infty}  \frac{\sup_{0\le k \le n} d(X_0^{\omega}, X_k^{\omega})}{n^{1/\beta} (\log \log n)^{1-1/\beta}} , \\
               \liminf_{ t \to \infty}   \frac{ \sup_{0 \le s \le t} d(Y_0^{\omega}, Y_s^{\omega}) }{t^{1/\beta} (\log \log t)^{-1/\beta}} 
                        &=  \left( \frac{1}{   \mathbb{E} \left[ \frac{1}{\pi^{\omega} (0)} \right]  } \right)^{1/\beta} 
                              \liminf_{ t \to \infty}  \frac{ \sup_{0 \le k \le n} d(X_0^{\omega}, X_k^{\omega}) }{n^{1/\beta} (\log \log n)^{-1/\beta}}  .
     \end{align*}
By Assumption \ref{Ass50}, Proposition \ref{Erg0-1} and Theorem \ref{MainLIL} we obtain Theorem \ref{Thm:Const}.


\begin{Acknowledgment}  
 This paper was written under the supervision of the author's Ph.D. advisor, Prof. Takashi Kumagai. 
 The author thanks him for suggesting me this problem, fruitful discussion and helpful comments.  
 The author also thanks Prof. Yuichi Shiozawa for meaningful discussion about this paper. 
 This research is partially supported by JSPS KAKENHI 15J02838.
\end{Acknowledgment}

\begin{flushleft}
Chikara Nakamura \\
Faculty of Science, Kyoto University. \\
Kitashirakawa Oimachi-cho, Sakyo ward, Kyoto city  \\
Japan 606-8224  \\
Email: chikaran@kurims.kyoto-u.ac.jp
\end{flushleft}

\end{document}